\documentclass[12pt]{amsart}
\usepackage{geometry}     
\usepackage{graphicx}
\usepackage{color}

\usepackage{amsthm}
\usepackage{amsmath}
\usepackage{amsfonts}
\usepackage{enumerate}
\usepackage{amssymb}
\usepackage{epstopdf}
\usepackage{hyperref}

\newtheorem{theorem}{Theorem}[section]
\newtheorem{proposition}[theorem]{Proposition}
\newtheorem{proposition*}{Proposition}
\newtheorem{lemma}[theorem]{Lemma}
\newtheorem{corollary}[theorem]{Corollary}

\DeclareMathOperator{\Hn}{\mathcal{H}^n}
\DeclareMathOperator{\Hnn}{\mathcal{H}^{n+1}}
\DeclareMathOperator{\G}{\Gamma}
\DeclareMathOperator{\Om}{\Omega}
\DeclareMathOperator{\Si}{\Sigma}

\newcommand{\R}{\mathbb{R}}

\newcommand{\Cgs}{5 \Hn(\partial U)}

\newcommand{\Clgs}{9 \Hn(\partial U)}

\theoremstyle{definition}
\newtheorem{definition}[theorem]{Definition}

\newtheorem{remark}[theorem]{Remark}

\title[Existence of minimal hypersurface]{Existence of minimal hypersurfaces
in complete manifolds of finite volume}

\author{Gregory R. Chambers}
\author{Yevgeny Liokumovich}

\begin{document}

\maketitle

\begin{abstract}
We prove that every complete non-compact
manifold of finite volume
contains a (possibly non-compact)
minimal hypersurface of finite volume.
The  main tool is the following result of 
independent interest: if a region $U$ can be swept out by a family of hypersurfaces of volume
at most $V$, then it can be swept out by a family of 
mutually disjoint
hypersurfaces of volume at most $V + \varepsilon$.
\end{abstract}

\section{Introduction}
\label{sec:Introduction}

By a result of Bangert and Thorbergsson 
(see \cite{Th} and \cite{Ba}) 
every complete
surface of finite area contains a closed
geodesic of finite length.
In this article we generalize this result
to higher dimensions.

Let $M^{n+1}$ be a complete Riemannian
manifold of dimension $n+1$.
For an open set $U \subset M$ define
the relative width of $U$, denoted by $W_{\partial}(U)$, 
to be the supremum over all real numbers $\omega$
such that every Morse function
$f:U \rightarrow [0,1]$
has a fiber of volume at least $\omega$.


\begin{theorem} \label{main}
Let $M^{n+1}$ be a complete Riemannian
manifold of dimension $n+1$, $(n+1) \geq 3$.
Suppose $M$ contains a bounded open set $U$
with smooth boundary, such that
$Vol_n(\partial U)\leq \frac{W_{\partial}(U)}{10}$.
Then $M$ contains a complete
embedded minimal hypersurface $\G$ of finite volume.
The hypersurface is smooth
in the complement of a closed set of Hausdorff dimension $n-7$.
\end{theorem}

\begin{remark}
We make some remarks about Theorem \ref*{main}:

1. The hypersurface $\G$ intersects a small neighbourhood of $U$.
In fact, for any $\delta>0$ there exists a finite area minimal hypersurface
that intersects the $\delta$-neighbourhood of $U$ (see Theorem \ref*{main_full}
and Question 3 in Section \ref*{questions}).

2. If $M$ is compact then $\G$ is compact. If $M$ is not compact
then $\G$ may or may not be compact.
In Remark \ref*{main_remark} we give an example,
suggesting that one can not always expect 
to obtain a compact minimal hypersurface
in a complete manifold of finite volume
using a min-max argument.

3. We also obtain upper and lower bounds for the volume
of $\G$ that depend on $U$ (see Theorem \ref*{main_full}).
\end{remark}

The condition that there exists a subset $U$ with
$\Hn(\partial U)\leq \frac{W_{\partial}(U)}{10}$ is 
satisfied if manifold $M$ has sublinear volume growth,
that is,
for some $x \in M$ we have $\liminf_{r \rightarrow \infty} \frac{Vol(B_r(x))}{r} =0 $ .
In particular, we have the following corollary.

\begin{corollary} \label{main'}
Every complete non-compact
Riemannian manifold $M^{n+1}$ of finite volume
contains a (possibly non-compact)
embedded minimal hypersurface of finite 
volume. The hypersurface is smooth
in the complement of a closed set of Hausdorff dimension $n-7$.
\end{corollary}

The proof is based on Almgren-Pitts min-max theory
\cite{Pi}.
We use the version of the theory developed by
De Lellis and Tasnady in \cite{DT}.
Instead of general sweepouts by integral flat cycles, the argument of \cite{DT} allows one to consider
sweepouts by hypersurfaces which are boundaries of open sets.
We consider a sequence of sweepouts of $U$ and extract a sequence
of hypersurfaces of almost maximal area that converges to a minimal hypersurface.
The main difficulty is to rule out the possibility that the sequence
completely escapes into the ``ends" of the manifold.
Proposition \ref*{prop:nested} is the main tool which allows us to rule out this possibility.
This Proposition allows us to replace an arbitrary family of hypersurfaces
with a nested family of hypersurfaces which are level sets of a Morse function,
increasing the maximal area by at most $\varepsilon$ in the process.

For closed manifolds Proposition \ref*{prop:nested} implies the
following result of independent interest.

\begin{theorem}
Let $M$ be a closed Riemannian manifold with min-max width $W(M)>0$.
For every $\varepsilon>0$ there exists a Morse function
$f: M \rightarrow \R$, such that the area of $f^{-1}(t)$
is bounded from above by $W+ \varepsilon$ for all $t \in \R$.
\end{theorem}

The precise definition of the width $W(M)$ is given in Section 3.

We use Proposition \ref*{prop:nested} together with some hands on geometric constructions
to show that there exists a sequence of hypersurfaces that converges to 
a minimal hypersurface and the volume of their intersection with a small neighbourhood of $U$
is bounded away from $0$.

A number of results about existence of minimal hypersurfaces
in non-compact manifolds have appeared recently.
Existence results for minimal hypersurfaces 
(compact and non-compact) in 
certain classes of complete non-compact manifolds
were proved by Gromov in \cite{Gr}.
This work was in part inspired by arguments in \cite{Gr}.
In \cite{Gr} mean curvature of boundaries
plays an important role.
Our results do not depend on the curvature of the 
manifold or mean curvature of hypersurfaces in $M$.
Existence of a compact embedded minimal surface in a hyperbolic 3-manifolds of finite volume
was proved by Collin-Hauswirth-Mazet-Rosenberg
in \cite{CHMR}.
In \cite{Mo} Montezuma gave a detailed proof 
of the existence of embedded closed minimal hypersurfaces in non-compact 
manifolds containing a bounded open subset with mean-concave boundary, 
as well as satisfying certain conditions
on the geometry at infinity.
In particular, these manifolds have infinite volume.
In \cite{KZ} Ketover and Zhou proved a
conjecture of Colding-Ilmanen-Minicozzi-White
about the entropy of closed surfaces in $\mathbb{R}^3$
using a min-max argument for the Gaussian area functional
on a non-compact space.
Finally, in \cite{So2} Song used min-max constructions of minimal
hypersurfaces in non-compact manifolds with cylindrical ends
to prove Yau's conjecture on the existence of infinitely many
minimal hypersurfaces in closed manifolds.

\bigskip

\noindent
\textbf{Acknowledgements} 
This paper uses, in a crucial way,
ideas from the work of Regina Rotman and the first author. The authors are grateful to 
Regina Rotman for many helpful discussions.

The second author would like to thank Camillo De Lellis and
Andr\'{e} Neves for organizing the Oberwolfach seminar ``Min-Max
Constructions of Minimal Surfaces" and for many fruitful
discussions there.
The authors would also like to thank Fernando Coda Marques and Andr\'{e} Neves for making several important suggestions.


The authors are grateful to Fernando Coda Marques
for pointing out two errors in an earlier 
draft of this paper. We are grateful to the referee for 
correcting several errors and making countless suggestions
that helped to improve the exposition.

The first author was partly supported by an NSERC postdoctoral fellowship.
The second author was partly supported by NSF grant DMS-171105.
This paper was partly written during the second
author's visit to Max Planck Institute at Bonn;
he is grateful for the Institute's kind
hospitality.

\section{Structure of proof} \label{sec:structure}

We describe the idea of the proof.  

\subsection{Families of hypersurfaces and sweepouts.}
In this article we will be dealing with families 
of possibly singular hypersurfaces $\{ \G_t \}$.
For the purposes of the introduction
the reader may assume that each $\G_t$
is a boundary of a bounded open set $\Om$
and has only isolated singularities of Morse type.
In fact, $\G_t$ may differ from $\partial \Om_t$ by a finite set of points.
The precise definition of the hypersurfaces
and the sense in which the family $\{ \partial \Om_t \}$
is continuous are described in Section \ref*{preliminaries}.
To follow the outline of the proof we only
need to know that the areas of $\partial \Om_{t_i}$
approach the area of $\partial \Om_{t}$
 and the volumes
of $(\Om_{t_i} \setminus \Om_{t}) \cup (\Om_{t} \setminus \Om_{t_i})$ 
go to zero as $t_i \rightarrow t$.
(We will use the word ``volume" for the $(n+1)-$dimensional 
Hausdorff measure and ``area" for the $n$-dimensional Hausdorff measure.)

We will consider four types of special families of
hypersurfaces, which we will call ``sweepouts".
We will study the relationship between these four types of
families and that will eventually lead us to the proof of
Theorem \ref*{main}.
Slightly informally we describe them below.

1. An (ordinary) \textbf{sweepout} of a bounded set $U$
is a family of hypersurfaces $\{ \partial \Om_t \}_{t \in [0,1]}$
with $\Om_0 \cap U = \emptyset$ and $U \subset \Om_1$.

2. A \textbf{good sweepout} of $U$ is a sweepout $\{\G_t\}$ with
areas of $\G_0$ and $\G_1$ less than $\Cgs$.

The motivation for this definition is the following. 
In a ``mountain pass" type argument
we would like to apply a ``pulling tight" deformation 
to a family $\{ \G_t \}$
so that hypersurfaces that have maximal area in the family
converge (in a certain weak sense) to a stationary point
of the area functional.
When doing this we would like hypersurfaces at the ``endpoints"
$\G_0$ and $\G_1$ to stay fixed.
We will consider sweepouts of sets with
the property that every sweepout must contain
a hypersurface of area much larger than 
the area of the boundary of $U$
(see definition of a good set below).
The condition above guarantees that
$\G_0$ and $\G_1$ do not have areas close to the maximum
and so the pulling tight deformation will not affect them.


3. A \textbf{nested sweepout} of $U$ is a sweepout 
$\{ \partial \Om_t \}_{t \in [0,1]}$ with
$\Om_s \subset \Om_t$ for every $s \leq t$.
Moreover, we have $\partial \Om_t = f^{-1}(t)$
for some Morse function $f$.
Nested sweepouts are a key technical tool
in this paper.

4. A \textbf{relative sweepout} of $U$ is a family of 
hypersurfaces $\{ \Si_t \}$ with boundaries $\partial  \Si_t \subset \partial U$
obtained from some nested sweepout $\{\G_t \}$ of $U$ 
by intersecting 
$\G_t$ with the closure of $U$, $\Si_t = \G_t \cap cl(U)$.

\subsection{Widths}
For each notion of a sweepout we define a corresponding notion of \textbf{width}.
If $\mathcal{S}$ is a collection of families of hypersurfaces we set
$$W(\mathcal{S}) = \inf_{\{\G_t \} \in \mathcal{S}} \sup_t \Hn (\G_t)$$

Let $\mathcal{S}(U)$, $\mathcal{S}_{\partial}(U)$, $\mathcal{S}_g(U)$ and $\mathcal{S}_n(U)$
denote the collection of all sweepouts, relative sweepouts, good sweepouts
and nested sweepouts correspondingly.
We set $W(U)=W(\mathcal{S}(U))$ to be the width of $U$,
$W_{\partial}(U)=W(\mathcal{S}_{\partial}(U))$ to be the relative width of $U$,
$W_g(U)=W(\mathcal{S}_g(U))$ to be the good width of $U$ and
$W_n(U)=W(\mathcal{S}_n(U))$ to be the nested width of $U$.

Theorem \ref*{main} is a statement about a bounded open set $U \subset M$
with smooth boundary and the property that $\Hn (\partial U ) \leq \frac{1}{10} W_{\partial}(U)$.
A set satisfying this property will be called a \textbf{good set}.
We will show that for a good set $U$
we have the following relationships
between the quantities $W(U)$, $W_{\partial}(U)$,
$W_g(U)$ and $W_n(U)$:

\begin{equation} \label{w-wrel}
    W_{\partial}(U) \leq W_n(U) \leq W_{\partial}(U) + \Hn(\partial U)
\end{equation}

\begin{equation} \label{w-wn}
    W_n(U) = W(U)
\end{equation}

\begin{equation} \label{w-wg}
    W_g(U) = W(U)
\end{equation}

The first inequality in (\ref*{w-wrel}) follows
directly from the definition.
The reason for the second inequality in (\ref*{w-wrel})
is also clear:
to obtain a nested sweepout $\{\G_t \}$
from a relative sweepout $\{\Si_t \}$
we can take a union of $\Si_t = \G_t \cap cl(U)$
with a subset of the boundary $\partial U$ (the subset varying based on $\Si_t$).
Certain perturbation arguments will guarantee 
that a sufficiently regular nested sweepout can 
be obtained in this way.
Note that this is also a good sweepout
since it starts on a hypersurface
of area $0$ and ends on a hypersurface of 
area $\Hn(\partial U) < \Cgs$.

Equation (\ref*{w-wn}) is proved in Proposition \ref*{prop:nested}.
In fact, (\ref*{w-wn}) holds not only for good sets $U$,
but for any bounded open set $U$ with smooth 
boundary. 
The proof of (\ref*{w-wn}) is the most technical 
part of this paper.

Equation (\ref*{w-wg}) is proved below using methods from
Section \ref*{sec: no escape}.
The importance of these equations is the following:
we will use (\ref*{w-wrel})
 and (\ref*{w-wn}) to prove (\ref*{w-wg});
 we will use (\ref*{w-wg}) and some of its slightly technical generalizations to prove
Theorem \ref*{main}.

\subsection{Existence of a large slice intersecting $U$.}
Now we can outline the proof of Theorem \ref*{main}.
We would like to find a minimal hypersurface in $M$ using a min-max argument, 
developed by Almgren \cite{Al} and Pitts \cite{Pi} and 
simplified by De Lellis - Tasnady \cite{DT}.
Let $U$ be a good set.
We choose a sequence of good sweepouts of $U$ with the property
that the area of the largest hypersurface converges to $W_g(U)$.
We would like to extract an appropriate sequence of hypersurfaces whose areas converge
to $W_g(U)$, and argue that they converge (as varifolds) to a minimal hypersurface.

The problem with this argument as it stands is that this sequence 
of hypersurfaces may drift off to infinity,
and so strong convergence may not hold.  
 To handle this issue, we will argue that this sequence
of hypersurfaces can be chosen so that the intersection of every hypersurface 
with $U$ is bounded away from $0$.  
This ``localization" statement will allow us to conclude
that in the limit we obtain a minimal hypersurface with non-empty support in a small neighbourhood of $U$.

\begin{proposition}
	\label{prop:intersection_U}
	For every good set $U$ there exists a positive constant $\varepsilon(U)$ which depends only on $U$ such that the following holds.
	For every good sweepout $\{ \Gamma_t \}$ of $U$ with associated family of open sets $\{ \Om_t \}$, 
	there is a surface $\Gamma_{t'}$ in the collection
	which has area at least $W_g(U)$, and such that $\Hn( \Gamma_{t'} \cap cl(U) ) \geq \varepsilon(U)$.
\end{proposition}

Theorem \ref*{main} will follow  
by modifying arguments in \cite{DT} (see Section \ref*{sec: convergence}).
In the remainder of this section we focus 
on the proof of Proposition \ref*{prop:intersection_U}.

We explain how we choose $\varepsilon(U)$.
In Section \ref*{sec: no escape} (Lemma \ref*{epsilon}) we will show that for every $U$ there exists 
$\varepsilon_0>0$ with the property that every $\Om$
which intersects $U$ in volume at most $\varepsilon_0$
or contains all of $U$ except for a set of volume at most $\varepsilon_0$
can be deformed so that its boundary does not intersect $U$ 
and the areas of the boundaries in the deformation process are controlled.
Specifically, if $\Hnn(\Om \cap U) \leq \varepsilon_0$ then there exists a family
$\{ \Om_t \}_{t \in [0,1]}$, such that
$\Om_0 \cap U = \emptyset$ and $\Om_1 = \Om$;
if $\Hnn(U \setminus \Om) \leq \varepsilon_0$
then there exists a family
$\{ \Om_t \}_{t \in [0,1]}$, such that
$\Om_1 \cap U = U$ and $\Om_0 = \Om$. In both cases the 
boundaries of $\Om_t$ satisfy

\begin{equation} \label{pushing out small set}
    \Hn(\partial \Om_t) < \Hn(\partial \Om) + 2 \Hn(\partial U)
\end{equation}

Having fixed $\varepsilon_0$ with this property
we define $\varepsilon(U) = \varepsilon(\varepsilon_0)>0$
to be such that every $\Om$ with
$\min \{\Hnn(\Om \cap U), \Hnn(U \setminus \Om)\} \geq \varepsilon_0/2$
has $\Hn(\partial \Om \cap U) > \varepsilon(U)$.
Existence of such $\varepsilon$ follows from the properties
of the isoperimetric profile of $U$.
In addition, we also require that $\varepsilon \leq \frac{\Hn(\partial U)}{5}$.


Suppose now that Proposition \ref*{prop:intersection_U} fails for this value of $\varepsilon(U)$.  
Let $V(t) = \Hnn(\Om_t \cap U)$ and
$A(t) = \Hn(\partial \Om_t \cap U)$. $V$ is a continuous function of $t$, $t \in [0,1]$, but $A(t)$ may not be continuous.
However, the family $\{ \partial \Om_t \}$
can be perturbed to make
$A(t)$ continuous.
In the proof of Proposition \ref*{prop:intersection_U} in Section \ref*{sec: no escape}
we prove a weaker assertion that
$A(t)$ is ``roughly" continuous after a small perturbation,
in the sense that the oscillation of $A$ at a point $t$ is at most $\varepsilon/10$;
this turns out to be sufficient for what we need.
For the purposes of this overview we will assume that $A(t)$ is actually continuous.

Continuity of $A$ and $V$ and the fact that $\{ \partial \Om_t \}$ is a sweepout
imply that there exists an interval $[a,b] \subset [0,1]$
with $\Hn(\partial \Om_t \cap U) \geq \varepsilon$ for all $t \in [a,b]$;
$\Hn(\partial \Om_a \cap U) = \varepsilon$ and $\Hn(\partial \Om_b \cap U) = \varepsilon$; 
$\Hnn(\Om_a \cap U) < \varepsilon_0/2$ and $\Hnn(\Om_b \cap U) > \Hnn(U)-\varepsilon_0/2$.
By our assumption this implies
$\Hn(\partial \Om_t) < W_g(U)$ for all $t \in [a,b]$.
Since $\Hn(\partial \Om_t)$ is a continuous function of $t$
there exists a real number $\delta > 0$
such that $\partial \Om_t$ has area at most $W_g(U) - \delta$ for $t \in [a,b]$.

Let $\tilde{U}= U \cap (\Om_b \setminus cl(\Om_a))$.
The last paragraph implies that
$W(\tilde{U}) \leq W_g(U) - \delta$.
The boundary of $\tilde{U}$ satisfies
$\Hn( \partial \tilde{U}) \leq \Hn( \partial U) + 2 \varepsilon$.

Equality (\ref*{w-wg}), whose proof is outlined below,
implies that a sweepout
of a good set can be replaced with a good sweepout,
while increasing the maximal area by an arbitrarily small amount.
In Section 7 we prove a more general result that 
if $U$ is a good set and $\tilde{U}$ is a
a subset of $U$ whose volume and boundary area are
sufficiently close to that of $U$, then every sweepout
of $\tilde{U}$ can be upgraded to a good sweepout.


It follows then that
$W_g(\tilde{U}) = W(\tilde{U}) \leq W_g(U) - \delta$
and, hence, there exists a good sweepout 
$\{\partial \tilde{\Om}_t \}_{t \in [0,1]}$ of $\tilde{U}$
with areas of all hypersurfaces at most $W_g(U) - \delta/2$.
By the definition of a sweepout 
$\tilde{\Om}_0 \cap U \subset U \setminus 
\tilde{U} \subset (U \cap \Om_a ) \cup (U \setminus \Om_b)$
and hence $\Hnn(\tilde{\Om}_0 \cap U) \leq \varepsilon_0$.
Also, since $\{\partial \tilde{\Om}_t \}$
is a good sweepout, $\partial \tilde{\Om}_0$ has
area at most $\Cgs + 10 \varepsilon$.
By (\ref*{pushing out small set}) we can deform
$\tilde{\Om}_0$ to a set that does not intersect $U$ 
through open sets with boundary area at most 
$W_g(U) - \delta/4$.
Similarly, we can deform $\tilde{\Om}_1$
to an open set that contains $U$
through open sets with boundary area at most 
$W_g(U) - \delta/4$.
We conclude that there exists a sweepout of $U$
by hypersurfaces of area at most 
$W_g(U) - \delta/4$. Hence, $W(U) \leq W_g(U) - \delta/4$,
which contradicts (\ref*{w-wg}). This finishes the proof
of Proposition \ref*{prop:intersection_U}.

\subsection{The good width equals width.}
In the rest of this section we 
describe how (\ref*{w-wg}) follows from (\ref*{w-wrel}) and (\ref*{w-wn}).  The argument is illustrated in Figure \ref*{fig:localization}.
We start with a sweepout $\{ \partial \Om_t \}$ of a good set $U$ by hypersurfaces
of area at most $W(U) + \delta$.
By (\ref*{w-wn}) we can assume that $\{ \partial \Om_t \}$
is a nested sweepout.
Next, we argue (cf. Lemma \ref*{long slice}) that there is a hypersurface $\partial \Om_{t'}$
with $t' \in [0,1]$ such that $\Hn(\partial \Om_{t'} \setminus U)$ has area
comparable to that of the boundary of $U$.
Indeed, by (\ref*{w-wrel}) there is a hypersurface with a large intersection
with $U$, that is, $\Hn(\partial \Om_{t'} \cap cl(U)) \geq W_n(U) - \Hn(\partial U)$.
The complement then must satisfy 
$\Hn(\partial \Om_{t'} \setminus cl(U)) \leq W(U)-W_n(U) + \Hn(\partial U) + \delta = \Hn(\partial U) + \delta$.

\begin{figure} 
	\centering
	\includegraphics[scale=0.7]{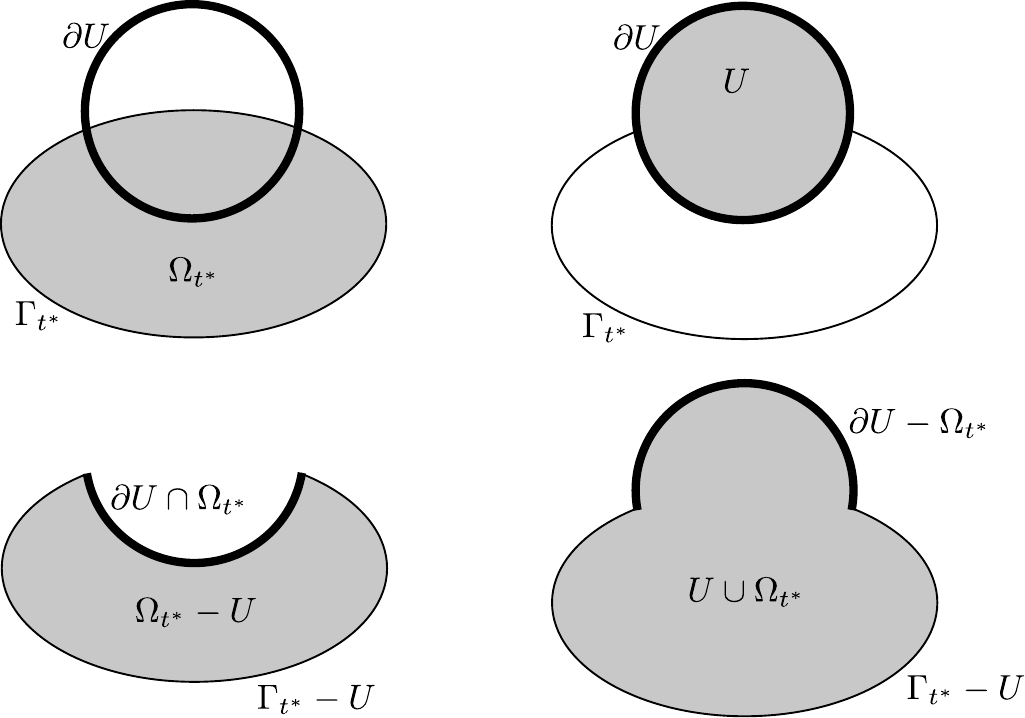}
	\caption{Cut and paste argument in the proof
	of (\ref*{w-wg})}
	\label{fig:localization}
\end{figure}

Now consider $\Om_{t'} \setminus U$. 
Since $\{ \partial \Om_t \}$ is nested this set contains
$\Om_0$ and is contained in $\Om_1$.
By the argument in the previous paragraph 
we have $\Hn(\partial (\Om_{t'} \setminus U)) \leq 2 \Hn(\partial U) + \delta$.
Let $A$ denote the infimal value of $\Hn (\partial \Om)$
over all open sets $\Om$ with $\Om_0 \subset \Om \subset \Om_{t'} \setminus U$.
Since $\Om_{t'} \setminus U$ is one of such sets we have
$$A \leq 2 \Hn(\partial U) + \delta$$

Let $\tilde{\Om}$ denote a set as above with 
$\Hn (\partial \tilde{\Om}) \leq A+ \delta$.
We replace sweepout $\{ \partial \Om_t \}$
with a new sweepout $\{ \partial (\tilde{\Om} \cup \Om_t) \}$.
Perturbation arguments will guarantee that we can smooth out the 
corners of these hypersurfaces to obtain a sufficiently regular family.
This family starts on a surface $\partial \tilde{\Om}$
of area less than $\Cgs$ 
and ends on $\Om_1$.
Moreover, it follows from the fact that $\partial \tilde{\Om}$ is $\delta$-nearly area minimizing
hypersurface that the area of $\partial (\tilde{\Om} \cup \Om_t)$
is bounded by $W+2 \delta$ (cf. Lemma \ref*{lem:split}).

Similarly, we can replace this sweepout with a new sweepout
that ends on a hypersurface of area less than $\Cgs$,
without increasing the areas of other hypersurfaces by more
than $\delta$. We conclude that $W_g(U) \leq W(U) + 3 \delta$,
but since $\delta>0$ was arbitrary (\ref*{w-wg}) follows.

The importance of nested sweepouts comes from the fact that it allows us to choose
nearly minimizing hypersurfaces like $\partial \tilde{\Om}$ and perform cut and paste procedures
as above without increasing the area significantly.
The ideas used in the proof of (\ref*{w-wn}) 
go back to \cite{CR} 
by the first author and Regina Rotman.  In that article, the authors were interested in
nested homotopies of curves, whereas here we use sufficiently regular cycles.  

\subsection{Open questions} \label{questions}
We list some open questions related to Theorem \ref*{main}.

1. For a positive real number  $\alpha$ 
we say that $U$ is an $\alpha$-good set if
$\Hn(\partial U) \leq \alpha W_{\partial}(U)$.
Theorem \ref*{main} asserts that if a complete
manifold $M$ contains a $\frac{1}{10}$-good set,
then there is a minimal hypersurface of finite
volume in $M$ which intersects a small neighbourhood of $U$.

It seems possible to improve the value of $\alpha$ to 
$\frac{1}{2}-\varepsilon$ using the methods of this paper by proving a sharper version
of Lemma \ref*{epsilon} and with more careful estimates in several other places.

\textit{Question:} What is the maximal value of $\alpha$
for which the conclusion of Theorem \ref*{main} holds?
It is conceivable that it may be true
for every positive $\alpha<1$.

2. In \cite{MN2} Marques and Neves show that a min-max
minimal hypersurface has a connected component of 
Morse index 1,
assuming that the manifold has no one-sided
hypersurfaces
(see \cite{MR}, \cite{So1}, \cite{Zh1}, \cite{Zh2} for previous results in that direction).
Is it possible to adapt their arguments
to construct a minimal hypersurface of finite volume and Morse
index 1 for every complete manifold without one-sided hypersurfaces and
satisfying the assumptions
of Theorem \ref*{main}?

3. In Theorem \ref*{main_full} we show that for an arbitrarily small
$\delta>0$ there exists a minimal hypersurface of finite
volume intersecting the $\delta$-neighbourhood of a good
set $U$. Does there exist a minimal hypersurface
of finite volume intersecting $cl(U)$? It is plausible that 
this result follows from a refinement of some of the arguments
in Section \ref*{sec: convergence} or from an appropriate compactness argument.

4. In \cite{Gr} it is shown that if a non-compact manifold
$M$ does not admit a proper Morse function $f$,
such that all non-singular fibers of $f$
are mean-convex, then $M$ contains a minimal
hypersurface of finite volume.
The following question was suggested to us by Misha Gromov:

\textit{Question:} Do there exist
manifolds of finite volume
that admit a 
Morse function $f$, such that all non-singular
level sets of $f$ have positive mean curvature?

More generally, do there exist good sets $U$ (in the sense
defined in this paper) which admit Morse foliations
by mean convex hypersurfaces (with boundaries of
the hypersurfaces
contained in the boundary of $U$)?

\section{Preliminaries} \label{preliminaries}

We begin with fixing notation and 
introducing several technical definitions which we will use throughout this article.

\begin{tabular}{ l l }
  $\mathcal{H}^k$ & $k-$dimensional Hausdorff measure  \\
  $cl(U)$ & closure of the set $U$ \\
  $B_r(x)$ & open ball of radius $r$ centered at $x$ \\
  $N_r(U)$ & the set $\{x \in M: d(x,U)< r\}$  \\
  $An(x,t_1,t_2)$ & the open annulus $B_{t_2}(x) \setminus cl( B_{t_1}(x))$
\end{tabular}

\vspace{0.2in}

Following De Lellis - Tasnady we make the following definitions.

\subsection{Families of hypersurfaces
and sweepouts}

\begin{definition} \label{hypersurface_def}
\textbf{Family of hypersurfaces}
A family $\{\Gamma_t\}$, $t\in [0,1]$,
of closed subsets of $M$ with
finite Hausdorff measure 
will be called a family of hypersurfaces if:

\vspace{0.2in}

(s1) For each $t$ there is a finite set 
$P_t \subset M$ such that $\Gamma_t$ is a smooth
hypersurface in $M \setminus P_t$;

(s2) $\Hn(\Gamma_t)$ depends smoothly 
on $t$ and $t \rightarrow \Gamma_t$
is continuous in the
Hausdorff sense;

(s3) on any $U \subset \subset M \setminus P_{t_0}$, $\Gamma_t \rightarrow \Gamma_{t_0}$
smoothly in $U$ as $t \rightarrow t_0$;

(s4) (no concentration of mass) for every point $x \in M$ 
we have $\limsup_{r \rightarrow 0} \sup_{t \in [0,1]}  \Hn(\G_t \cap B_r(x)) = 0$.
\end{definition}

\begin{definition} \label{def: sweepout}
\textbf{Sweepout}
Let $U$ be an open subset of $M$.
$\{\Gamma_t \}$, $t \in [0,1]$, is a sweepout of $U$ 
if it satisfies (s1)-(s4) and there exists a family $\{  \Om_t\}$, 
$t \in [0,1]$, of open
sets of finite Hausdorff measure, such that

\vspace{0.2in}

(sw1) $(\Gamma_t  \setminus \partial   \Om_t)
\subset P_t$ for any $t$;

(sw2) $  \Om_0 \cap U = \emptyset$ and 
$U \subset   \Om_1$;

(sw3) $\Hnn(  \Om_t \setminus   \Om_s) + \Hnn(\Om_s \setminus   \Om_t) 
\rightarrow 0$ as $t \rightarrow s$.

\vspace{0.2in}

For a sweepout $\{ \Gamma_t \}$ 
we will say that $\{   \Om_t \}$
is the corresponding family of open sets 
if it satisfies (sw1) - (sw3).

\end{definition}

\begin{definition}
\textbf{Good sweepouts, nested sweepouts
and relative sweepouts}

A \textbf{good sweepout} $\{ \G_t \}$
is a sweepout of $U$ which in addition satisfies:

(sw$_g$) $\Hn(\Gamma_0)\leq \Cgs$ and 
$\Hn(\Gamma_1) \leq  \Cgs$.

\vspace{0.2in}

A \textbf{nested sweepout} $\{ \G_t \}$
is a sweepout of $U$ which in addition satisfies:

(sw$_n$) there exists a Morse function 
$f: M \rightarrow [-1, \infty)$,
such that $\G_t = f^{-1}(t)$, $t \in [0,1]$;
the corresponding family of open sets is given by
$\Om_{t} = f^{-1}((-\infty, t))$.

\vspace{0.2in}

Suppose $\partial U$ is a smooth manifold
and  $\{ \G_t \}$ is a nested sweepout of $U$ with 
the corresponding family of open sets $\{\Om_t\}$.
Set $\Si_t = (cl(U) \cap \G_t)$.
We will say that $\{ \Si_t  \}$ is a
\textbf{relative sweepout} of $U$.

\end{definition}

\begin{definition}
\textbf{Widths and good sets}
As described in Section \ref*{sec:structure}
the widths $W(U)$, $W_{\partial}(U)$,
$W_g(U)$ and $W_n(U)$ are defined as the min-max quantities 
corresponding to sweepouts, relative sweepouts, good sweepouts
and nested sweepouts respectively.

A good set $U \subset M$ is a bounded open set with smooth boundary
and $\Hn(\partial U) \leq \frac{1}{10} W_{\partial}(U)$.
\end{definition}

\subsection{Smoothing corners.}  \label{smoothing}

Let $N \subset M$ be an open subset
and suppose $\Si_1 \subset \partial N$ and 
$\Si_2 \subset \partial N$ are $n$-dimensional
submanifolds of $M$, such that
the interiors of $\Si_1$ and $\Si_2$ are disjoint,
$\Si_1 \cup \Si_2 = \partial N$ and 
$\partial \Si_1 \cap \partial \Si_2 = C$ is 
a compact $(n-2)$-dimensional submanifold
of $M$.

We say that $\partial N$ is a manifold
with corner $C$ if for every sufficiently
small neighbourhood $U$ of a point $x \in C$
there exists a diffeomorphism $\phi$
from $U$ to $\mathbb{R}^{n+1}$ with
$\phi(N)= \mathbb{R}_+ \times \mathbb{R}_+ \times 
\mathbb{R}^{n-1}$,
$\phi(\Si_1) = \{x_1 = 0 \}$,
$\phi(\Si_2) = \{x_2 = 0 \}$
and $C = \{x_1 = x_2 = 0 \}$.

There is a standard construction
of smoothing (or straightening) the corner 
$C$ of a manifold with corner (see \cite[Section 7.5]{Mu}).
We briefly describe it here,
because we use it several times
in this paper.

Fix $\delta>0$. We construct a smooth hypersurface $\Sigma  \subset cl(N)$,
such that $\Sigma$ coincides with $\partial N$
outside of $N_{\delta}(C)$. 

Define cylindrical coordinates $y=(x, \theta, r)$ on $cl(N_{\delta}(C) \cap N)$,
where $x \in C$, $r$ denotes the distance to $C$ and
$\theta \in S^1$ is the angle. 
We make a choice of coordinates so that for a fixed $x$ 
geodesic ray $t_x^1= \{\theta = 0 \}$ is tangent to $\Sigma_1$ at $x$
and geodesic ray $t_x^2(r) = \{ \theta = \alpha(x)\}$ is tangent to $\Sigma_2$ 
at $x$ for some smooth function $\alpha(x)$.
In these coordinates two-dimensional disc 
$D_x(\delta)=\{(x, \theta, r): 0 \leq r \leq \delta, 
\theta \in S^1\}$ intersects $\Sigma_1$ and $\Sigma_2$ in two curves, 
$s_x^1$ and $s_x^2$ correspondingly, meeting at the point $(x,0,0)$, and $cl(N) \cap D_x$ is 
the region bounded by $s_x^1 \cup s_x^2$.
We may assume that $\delta>0$ is sufficiently small,
so that the following two conditions are satisfied
for $i=1,2$ and all $x \in C$:

a)  every tangent line to $s_x^i$ in $D_x(\delta)$
meets $t_x^i$ at an angle at most $\frac{\pi}{8}$;

b)  $s_x^i$ does not intersect geodesic rays $\theta = \frac{\alpha(x)}{8}$
and $\theta = \frac{7\alpha(x)}{8}$.

Let $\gamma_x$ be a family of smooth curves in $D_x(\delta/2)$,
so that in $D_x(\delta/4)$ the curve is given by a smooth convex
function $r= r_x(\theta)$
and in $D_x(\delta/2) \setminus D_x(\delta/4)$
it coincides with the rays $\theta = \frac{\alpha(x)}{8}$
and $\theta = \frac{7\alpha(x)}{8}$.
We have that in the annulus $D_x(\delta/2) \setminus D_x(\delta/4)$
curve
$\gamma_x$ consists of two connected components each graphical
over $s_x^1$ and $s_x^2$ correspondingly.
It follows that we can extend $\gamma_x$ to $D_x(\delta)$
smoothly so that in the annulus $D_x(\delta) \setminus D_x(3\delta/4)$
curve $\gamma_x$ coincides with $s_x^1$ and $s_x^2$.

We make several observations about this construction.

1. Different smoothings $\Sigma$ corresponding to different choices 
of curves $\gamma_x$ in the above construction are all isotopic.

2. For any $\varepsilon> 0$ curves $\gamma_x$ can be chosen in such a way that 
$\Hn(\Sigma) < \Hn(\partial N) + \varepsilon$.

3. Smoothing can be done parametrically. Given a foliation of a subset of $M$ by hypersurfaces with corners
the above construction can be applied to the whole family in such a way that
we obtain a foliation by a family of smooth hypersurfaces.

4. For all $\delta>0$ sufficiently small there exists a choice
of $\Sigma$ and a constant $c$ that depends on $M$, $N$ and $C$, so that 
$\Hn(  N_{10 \delta}(C) \cap \partial N_{2 \delta}(\Sigma)) 
\leq c \delta$.

The last observation will be important in the proof of 
Lemma \ref*{gluing_separated}.

It will be convenient to introduce one more definition.

\begin{definition} \label{delta-perturbation}
Let $\Om \subset M$ be a bounded open subset 
and $\partial \Om$ is a manifold with corner
and $\delta>0$.
We will say that $\Om_{+ \delta}$ is an outward 
$\delta$-perturbation of $\Om$ if the following holds:

(1) $\Om \subsetneqq \Om_{+ \delta} \subset N_{\delta}(\Om)$;

(2) there exists a nested family of open sets
$\{\Xi_t\}_{t \in [0,1]}$ and a smooth isotopy $\Si_t = \partial \Xi_t$,
such that $\Si_0$ is a smoothing of $\partial \Om$,
$\Xi_1 = \Om_{+ \delta}$
and $\Hn(\Si_t) < \Hn(\partial \Om) +\delta$
for all $t \in [0,1]$.

We will say that $\Om_{- \delta}$ is an inward 
$\delta$-perturbation of $\Om$ if the following holds:

(1)' $\Om \setminus N_{\delta}(\partial \Om) \subset \Om_{-\delta} \subsetneqq \Om$;

(2)' there exists a nested family of open sets
$\{\Xi_t\}_{t \in [0,1]}$ and a smooth isotopy $\Si_t = \partial \Xi_t$,
such that $\Si_1$ is a smoothing of $\partial \Om$,
$\Xi_0 = \Om_{- \delta}$
and $\Hn(\Si_t) < \Hn(\partial \Om) +\delta$
for all $t \in [0,1]$.

\end{definition}

\section{Morse foliations with controlled
area of fibers.}


Here we present several results about
concatenating different Morse foliations and
controlling areas of fibers of
Morse functions.

For PL Morse functions Sabourau proved similar results in \cite{Sa}.

\subsection{Gluing Morse foliations.}

Let $N^{n+1} \subset M$ be a compact submanifold of $M$ 
with boundary.
We will say that a Morse function $f: N \rightarrow \mathbb{R}$ is
$\partial$-transverse if 

(1) there exists an extension $\bar{f}$ of $f$ to an open neighbourhood of $N$ in $M$, such that all critical points are 
isolated, non-degenerate and lie in the interior of $N$;

(2) the restriction of $f$ to $\partial N$
is a Morse function.

\begin{lemma} \label{morse}
Let $N^{n+1} \subset M^{n+1}$ be a compact submanifold 
with non-empty boundary and $f: N \rightarrow [a,b]$
be a $\partial$-transverse Morse function. Let $\Sigma^{n}$ be
a closed submanifold of $\partial N$.

For every $\varepsilon >0$ there exists a Morse function $g: N \rightarrow [a,b]$,
such that the following holds:

(1) 
$g^{-1}(b) = \Sigma$;

(2) $f^{-1}([a,t)) \subset N_{\varepsilon/2} (g^{-1}([a,t))) \subset N_{\varepsilon}(f^{-1}([a,t)))$;

(3) $ \Hn(g^{-1}(t)) \leq \Hn(\partial f^{-1}([a,t])) + \varepsilon$;

(4) If $dist(x, f(\Sigma))> \varepsilon$ then $f^{-1} (x) = g^{-1}(x)$.
\end{lemma}

\begin{proof}

The idea of the proof is shown in Figure \ref*{fig:slice}.

We will define a singular foliation  
$\Si_t$, $t \in [0,1]$, of $N$ with 
only finitely many singular leaves that
have non-degenerate singularities and with $\Si_1 = \Sigma$.
It follows then that there exists a Morse function
 $g(x)$ with $g^{-1}(t) = \Si_t$.
We will prove that this foliation satisfies the
desired upper bound on the area. 
The surfaces in the foliation will coincide with $f^{-1}(t)$
whenever $f^{-1}(t)$ is sufficiently far from 
$\Sigma$ and so (4) will also follow.

\begin{figure} 
	\centering
	\includegraphics[scale=0.75]{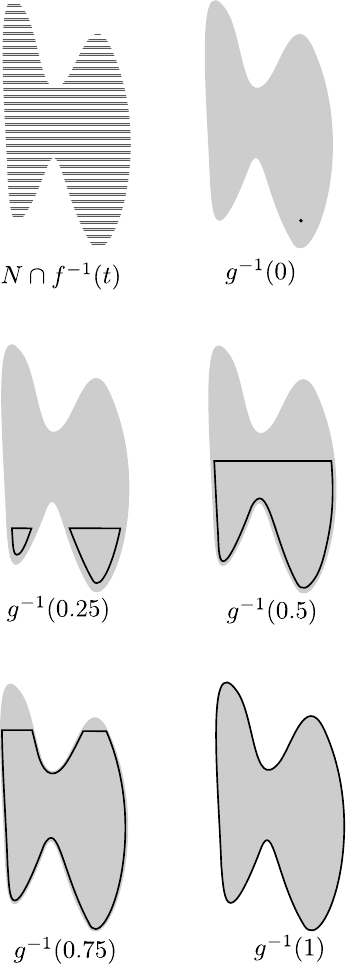}
	\caption{Constructing a singular foliation of $N$.}
	\label{fig:slice}
\end{figure}

Choose $r_0 \in (0, \varepsilon)$, be sufficiently small, so that
the tubular neighbourhood 
$U= N_{2 r_0}(\Sigma) \cap N$
does not intersect
critical points of $f$ and
there exists a diffeomorphism $\phi$ from 
$\Sigma \times [0,2 r_0)$ to
$U$. Let $\phi (x,r)$, $x \in \Sigma$, $r \in [0,r_0)$  
denote the Fermi coordinates on $U$.
For $r_0$ sufficiently small we may assume that $\Hn((\Sigma, r))\leq \Hn(\Sigma) + \frac{\varepsilon}{2}$ for $r \in [0,r_0]$.
Let $U_r = \{ \phi (x,r'): r' \leq r \}$.
Let $\varepsilon_0=\varepsilon_0(r_0)>0$ be a small constant to be specified
later and satisfying $\varepsilon_0 \rightarrow 0$
for $r_0 \rightarrow 0$.

Let $p_0<...<p_k$ be critical values 
of $f|_{\Sigma}$.  First we
define a singular foliation $\Sigma_t$, $t \notin \cup_i (p_i - \varepsilon_0, p_i+\varepsilon_0)$.
Let $\bar{\Si}_t = \partial (f^{-1}([a,t]) \setminus U_{(1-t) r_0 } )$.
If $t$ is a singular value of $f$ then $\bar{\Si}_t$ has a Morse type singularity
at the singular point $s$ of $f$ in the interior of $N$.
Since $t$ is at least $\varepsilon_0$ away
from singular values of $f|_{\Sigma}$ we have that
$f^{-1}(t)$ intersects $\phi(\Sigma, (1-t)r_0 )$ transversally.
Hence, $\bar{\Si}_t \setminus s$ is a manifold with corners.
There exists a smoothing of the corners, so that 
the new foliation $\{ \Sigma_t \}$ coincides with 
$\{ \bar{\Si}_t \}$ outside of a small neighbourhood 
of $V_t=f^{-1}(t) \cap \phi (\Sigma, (1-t) r_0)$ 
and is smooth in $V_t$.
As discussed in subsection
\ref*{smoothing} we can choose it 
so that $\Hn(\Sigma_t) - \Hn(\bar{\Sigma}_t)$
is arbitrarily small.

Now we construct the foliation for $t \in (p_i - \varepsilon_0, p_i+\varepsilon_0)$.
Let $x_i \in \Sigma$ be the critical point of
$f|_{\Sigma}$ with $f(x_i) = p_i$.
Outside of a small neighbourhood of $x_i$
we can define $\Sigma_t$ in the same way as above,
since $f^{-1}(t)$ intersects $\phi(\Sigma, (1-t)r_0)$ transversally
and a smoothing of the corners is well-defined.

In the neighbourhood of 
a critical point $x_i$ we define the foliation
by considering two cases (see Figure \ref*{fig:morse_cases}). 
Since $f$ is $\partial$-transverse we have that $\nabla f(x_i) \neq 0$.
Let $n_i$ denote the inward pointing 
unit normal at $x_i$ and set
$s_i = \frac{\langle f(x_i), n_i \rangle }{|\langle f(x_i), n_i \rangle |}$.
The two cases will depend on the sign of $s_i$.

Let $y_i = \phi (x_i, (1-p_i) r_0)$.
There exists a choice of coordinates $u=(u_1,...,u_{n+1})$
in the neighbourhood of $y_i$ so that
in these coordinates we have
$f(u) = u_{n+1} + f(y_i)$.
Let $\lambda$ denote the index of $x_i$.
Let $P_{\lambda}(u_1, ..., u_n) = -u_1^2 -...- u_{\lambda}^2
+ u_{\lambda+1}^2 + ... + u_n^2$.
Up to a bilipschitz diffeomorphism 
of the neighbourhood of 
$y_i$, the foliation 
$\{ \phi(\Sigma, (1- t')r_0) \}$,
$t' \in (p_i- \varepsilon_0, p_i + \varepsilon_0)$, will coincide
with the foliation $\{  u_{n+1} = P_{\lambda}(u_1,...,u_n) - s_i t \}$, $t \in (- \varepsilon_0, \varepsilon_0)$. 

\begin{figure} 
	\centering
	\includegraphics[scale=0.5]{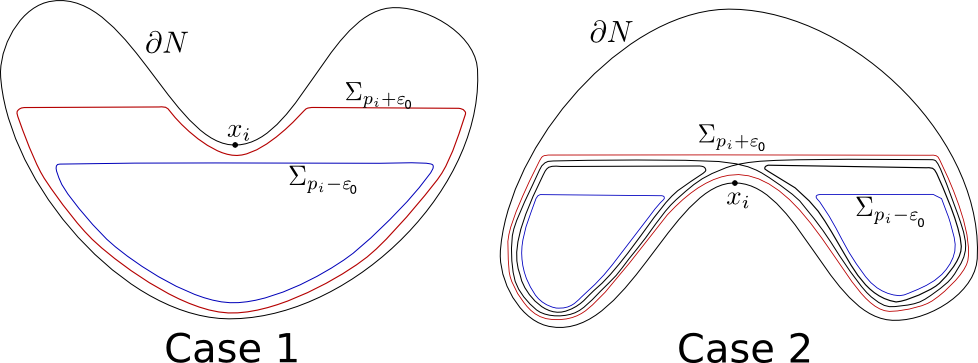}
	\caption{Procedure for dealing with singularities on the boundary of $N$.}
	\label{fig:morse_cases}
\end{figure}

Case 1: $s_i  =-1$.
There exists a smoothing of the corners for $\Sigma_t$
so that 
as $t$ approaches $p_i$ from 
above and below
surface
$\Sigma_t$ is a graph over $\{ u_{n+1} =0 \}$ hyperplane
in the neighbourhood of $y_i$. 
There exists a small $\delta >0$ and a foliation $\{ \G_t \}$
of the neighbourhood of $y_i$ so that
$\G_t = \{u_{n+1} = P_{\lambda}(u_1,...,u_n) + t \}$ 
for $u_1^2+...+u_n^2< \delta/3$ and
$\G_t$ is a graph of $u_{n+1} = t$ 
for $u_1^2+...+u_n^2> 2 \delta/3$.
The foliation $\{ \G_t \}$ extends the foliation
$\{ \Sigma_t \}$ to the neighbourhood of 
the critical point $x_i$.

Case 2: $s_i  =1$.
Let $\Pi_t= \{u_{n+1} = t \} \cap
\{P_{\lambda}(u_1,...,u_n) \leq 2t \}$
and $Q_t = \{u_{n+1} = P_{\lambda}(u_1,...,u_n) - t \} \cap
\{u_{n+1} \leq t \}$.
After a bilipschitz diffeomorphism in the neighbourhood
of $y_i$ we may assume that the foliation $\{ \Sigma_{t'} \}$
is given by the smoothing of the union
$\Pi_t \cup Q_t$.
By standard Morse theory arguments (see Section 3 of \cite{Mi1}
and Section 3 of \cite{Mi2}) $\Pi_{\delta} \cup Q_{\delta}$
is obtained from $\Pi_{-\delta} \cup Q_{-\delta}$
by surgery of type $(\lambda, n+1-\lambda)$ and
there exists an elementary cobordism between them
of index $\lambda$.
This cobordism gives the desired foliation in the neighbourhood 
of the critical point. 

Observe that in the above operations we applied
bilipschitz diffeomorphisms on some small neighbourhood,
possibly increasing the areas of hypersurfaces
by some controlled constant factor 
(independent of
the size of the neighbourhood). By choosing the neighbourhood 
to be sufficiently small we ensure that the areas 
do not increase by more than $\varepsilon$.


\end{proof}

We will also need a slightly different
version of this lemma for a non-compact 
submanifold $N$.

\begin{lemma} \label{morse_noncompact}
Let $N \subset M$ be a not necessarily compact submanifold 
with non-empty boundary and $f: N \rightarrow (-\infty,b]$
be a proper Morse function, which is $\partial$-transverse. Let $\Sigma$ be
a compact submanifold of $\partial N$.

For every $\varepsilon >0$ there exists a Morse function $g: N \rightarrow (-\infty,b]$,
such that the following holds:

(1) 
$g^{-1}(b) = \Sigma$;

(2) $f^{-1}((-\infty,t)) \subset N_{\varepsilon/2} (g^{-1}((-\infty,t))) \subset N_{\varepsilon}(f^{-1}((-\infty,t)))$;

(3) $ \Hn(g^{-1}(t)) \leq \Hn(\partial f^{-1}((-\infty,t])) + \varepsilon$;

(4) If $dist(x, f(\Sigma))> \varepsilon$ then $f^{-1} (x) = g^{-1}(x)$.
\end{lemma}

\begin{proof}
Let $a$ be such that $f(N_{\varepsilon} (\Sigma)) \subset [a+\varepsilon,b]$.
Since function $f$ is proper we have that $N' = f^{-1} ([a,b])$
is compact. We apply Lemma \ref*{morse}
to $N'$ to obtain function $g$. 
We set $g(x)=f(x)$ for $x$ not in $N'$ and the lemma follows.
\end{proof}

\subsection{Gluing Morse foliations 
on a manifold separated by a hypersurface
transverse to the boundary.}

We will also need the following lemma for
gluing two Morse foliations
on a manifold with boundary
separated by a hypersurface
which is transversal to the boundary.

\begin{lemma} \label{gluing_separated}
Let $N^{n+1} \subset M^{n+1}$ be a manifold with compact
boundary $\partial N$ and $\Sigma$ be
a hypersurface with $\partial \Sigma 
\subset \partial N$ and such that
$\Sigma$ intersects $\partial N$
transversally and separates $N$ into two disjoint
regions $N \setminus \Sigma = V_1 \sqcup V_2$.
For every $\varepsilon >0$ there exist 
open sets with smooth boundary
$\Om_1$ and $\Om_2$ and a Morse
function $g: cl(N \setminus (\Om_1 \cup \Om_2))
\rightarrow [0,1]$, such that the following
holds:

(1) $\Om_1$ is an inward $\varepsilon$-perturbation of $V_1$;
$\Om_2$ is an inward $\varepsilon$-perturbation of $V_2$.


(2) $g^{-1}(0) = \partial \Om_1 \cup \partial \Om_2$
and $g^{-1}(1) = \partial N$;

(3) $\Hn(g^{-1}(t)) \leq \Hn(\partial N) + 
2 \Hn(\Sigma) + \varepsilon$.

\end{lemma}

\begin{figure} 
	\centering
	\includegraphics[scale=0.5]{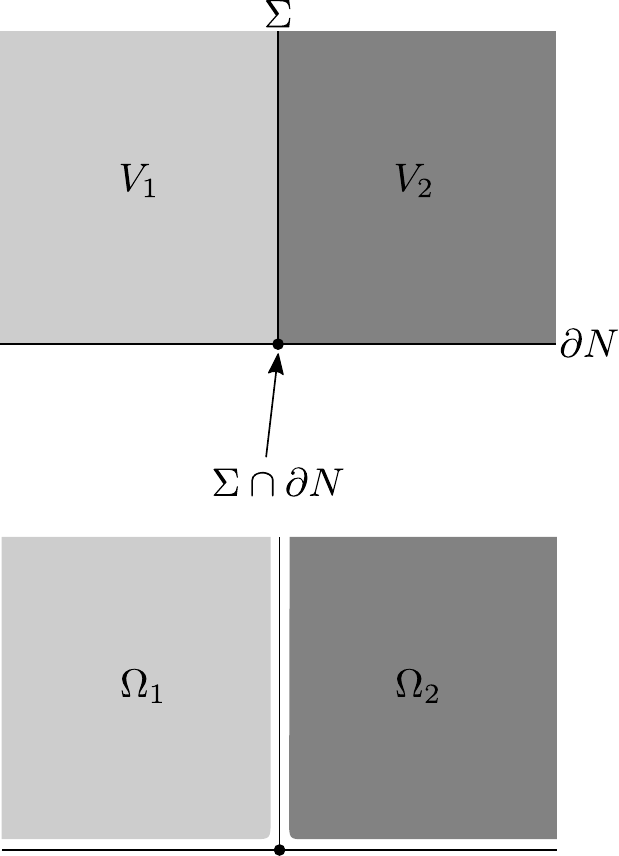}
	\caption{Gluing two submanifolds using a Morse foliation.}
	\label{fig:gluing_submanifold}
\end{figure}

\begin{proof}
The situation is illustrated in Figure \ref*{fig:gluing_submanifold}.

Let $f: M \rightarrow \R$ be a Morse function, which 
coincides with the signed distance function $sdist(x, \partial N)$ for $x \in N_{\delta_1}(\partial N)$
for some sufficiently small $0<\delta_1< \varepsilon$. 
The sign is chosen so that $f(x) \geq 0$ when $x \in N$.
Without any loss of generality we may assume that $f$ has no critical points in $\varepsilon$-neighborhood of $\Sigma$,
the restriction of $f$ to $\Sigma$ is also Morse and
$\Hn(f^{-1}(t)) \leq \Hn(\partial N)+ \varepsilon/10$ for $t \in [0, \delta_1]$.
Pick $\delta_2 \in (0, \delta_1)$, so that $\Hn(f^{-1}(x) \cap N_{\delta_2}(\Sigma)) < \varepsilon/10 $
and $\Hn(\partial N_{\delta_2}(\Sigma)) \leq 2 \Hn(\Sigma) + \varepsilon/10$.

Let $\Om_i$ denote an inward $\frac{\delta_2}{2}$-perturbation of $V_i$.
We can make an arbitrarily small perturbation to $f$ by an ambient diffeomorphism,
so that $f$ is a Morse function on $\partial \Om_1 \cup \partial \Om_2$
and for $t \in [0, \delta_2]$ each leaf of the foliation $f^{-1}(t) \cap N \setminus (\Om_1 \cup \Om_2)$
is a graph over some region $U_t \subset N$ with $\Hn(f^{-1}(t) \cap N \setminus (\Om_1 \cup \Om_2))\leq \Hn(U_t) + \varepsilon/5$.

It follows that $\Hn(\partial f^{-1}((-\infty, t)) \leq 2 \Hn(\Sigma) + \Hn(\partial N) + \varepsilon/2$.
We can now apply Lemma \ref{morse_noncompact} to $f$ restricted to $M \setminus (\Om_1 \cup \Om_2)$
to obtain desired function $g$ and the corresponding foliation.
\end{proof}

\section{Splitting and extension lemmas}

In this section we prove two important lemmas
for nested sweepouts which we will use in sections
``Nested sweepouts" and ``No escape to infinity".

    \begin{lemma}
	\label{lem:split}
	Suppose that $f: M \rightarrow [-1,\infty)$ is a Morse
	function and $\{ \Gamma_t \} = \{ f^{-1}(t) \}_{t \in [0,1]}$ is a 
	nested family of hypersurfaces of area $\leq A$ 
	with associated open sets $\{ \Om_t \} = 
	 \{ f^{-1}((-\infty, t)) \}$.
	
	I. Additionally, suppose that $\Om$ is a bounded open set 
	with boundary $\Gamma$ a smooth embedded manifold such that
	\begin{enumerate}
		\item 
		      $\Om \subset \Om_1$;
        \item There is an $\varepsilon > 0$ such that
                for every $\Om'$ with $\Om \subset \Om' \subset \Om_1$ we have
                $\Hn( \Gamma)  <  \Hn(\partial \Om') + \varepsilon/4$.
	\end{enumerate}

	Then we can find a nested family $\tilde{\G}_t$ and an associated
	family of open sets $\tilde{\Om}_t$ such that $\tilde{\Om}_0 \subset \Om_0$, $\tilde{\G}_1 = \Gamma$,
	and every hypersurface has area at most $A + \varepsilon$.
	Furthermore, if $\Om_0 \subset \Om$, then $\tilde{\G}_0 = \G_0$.
	
	II. Suppose that, instead of properties (1) and (2) above, the following are true:
	
		(1)' 
		      $\Om_0 \subset \Om$;
		      
		(2)'    There is an $\varepsilon > 0$ such that
                for every $\Om'$ with $\Om_0 \subset \Om' \subset \Om$ we have
                $\Hn( \Gamma)  <  \Hn(\partial \Om') + \varepsilon/4$.
		

	Then we can find a nested family $\tilde{\G}_t$ and an associated
	family of open sets $\tilde{\Om}_t$ such that $\Om_1 \subset \tilde{\Om}_1$, $\tilde{\G}_0 = \Gamma$,
	and every hypersurface has area at most $A + \varepsilon$.
	Furthermore, if $\Om \subset \Om_1$, then $\tilde{\G}_1 = \G_1$.

\end{lemma}
\begin{proof}
    We begin with a proof of the first half of this lemma.
    
    The desired family will be obtained 
    by regularization of the collection of hypersurfaces
    $\{ \partial (\Om_0 \cap \Om) \}$.
    
    We consider two cases. 
    Suppose first that $\Om \subset \Om_0$.
    For a sufficiently small $\delta>0$
    the function $g: cl(N_\delta(\G) \cap \Om) \rightarrow [0,1]$
    given by $g(x)= \frac{1}{\delta} dist(x, \G)$ is a smooth function
    with no critical points and $\tilde{\G}_t = g^{-1}(t)$ a hypersurface
    of area at most $\Hn(\G) + \varepsilon/2$. By condition (2)
    $\Hn(\G) \leq \Hn(\partial \Om_0) + \varepsilon/4$ and so
    $\Hn(\tilde{\G}_t) \leq A + \varepsilon$.
    We extend $g$ to a Morse function on $M$ in an arbitrary way.
    $\{ \tilde{\G}_t \}$ is a nested family satisfying the conclusions of the
    theorem.
    
    Suppose now that $\Om  \setminus \Om_0  \neq \emptyset$.
    Make a small perturbation to the hypersurface $\G = \partial \Omega$, so that
    $f|_{\partial \Omega}$ is Morse and (1) and (2) are still satisfied,
    possibly replacing $\varepsilon/4$ in (2) by $ \varepsilon/2$.
    
    Consider $f$ restricted to $\Om$.
    We apply Lemma \ref*{morse}
    with $N= \Om $ and $\Sigma = \Gamma$ to obtain
    a Morse function $g: \Om \rightarrow [-1,1]$, such that
    $g^{-1}(-1)$ is a point in $\Om$,
    $g^{-1}(1) = \G$ and  $\Hn(g^{-1}(t)) \leq \Hn(\partial (f^{-1}([-1,t])
    \cap \Om) + \varepsilon/2$.
    It follows that 
    $\Hn(g^{-1}(t)) \leq \Hn(f^{-1}(t) \cap \Om) + \Hn(f^{-1}([-1,t])\cap \G) + \varepsilon/2$.
    Furthermore, we have $g^{-1}([-1,0)) \subset N_{\varepsilon} (\Om \cap \Om_0)$.
    After a small perturbation of the function $g$ we may assume that
    $g^{-1}([-1,0)) \subset (\Om \cap \Om_0)$.
    We extend $g$ to a Morse function on $M$ in an arbitrary way.
    We claim that $\title{\G}_t = g^{-1}(t)$ for $t\in [0,1]$
    is the desired nested family. The only thing left to prove
    is an upper bound for the areas of $\tilde{\G}_t$.
    
    For any smooth hypersurface $\Si_t$ obtained 
    by a small perturbation of $\partial (\Omega \cup \Omega_t)$
    we have $\Hn(\G) \leq \Hn(\Si_t) + \varepsilon/4$ by (2).
    It follows that 
    $$\Hn(\G) \leq \Hn(\partial (\Omega \cup \Omega_t)) + \varepsilon/2$$
    
    Since $\partial (\Omega \cup \Omega_t)
    = (\G_t \setminus \Om) \cup (\G \setminus \Om_t)$
    we have
    $$\Hn(\G \cap \Om_t) + \Hn(\G \setminus \Om_t) \leq \Hn(\G_t \setminus \Om )
    + \Hn(\G \setminus \Om_t) + \varepsilon/2$$
    $$\Hn(\G \cap \Om_t) \leq \Hn(\G_t \setminus \Om ) + \varepsilon/2$$
    
    By Lemma \ref*{morse} we have

\begin{align*}
\Hn(\tilde{\G}_t) & \leq \Hn( \G_t \cap \Omega) + \Hn( \G \cap \Om_t) + \varepsilon/2 \\
                & \leq \Hn( \G_t \cap \Omega) + \Hn(\G_t \setminus \Om ) + \varepsilon \\
                & \leq \Hn(\G_t) + \varepsilon \leq A + \varepsilon
\end{align*}

If $\Om_0 \subset \Om$, then by choosing 
sufficiently small $\varepsilon>0$ and applying
Lemma \ref*{morse} (4)
we have $\tilde{\G}_0 = f^{-1}(0) = \G_0$.

The proof of the second half is similar.
The desired family will be a regularization of
$\{\partial (\Omega_t \cup \Omega) \}$.

If $\Om_1 \subset \Om$ we define the desired nested family
$\{ \tilde{\G} \}$ in a small tubular neighbourhood of 
$\G$.

Otherwise, 
define $\tilde{f}(x) = -f(x)$.
We apply Lemma \ref*{morse_noncompact} 
to the restriction  $\tilde{f}: M \setminus \Om \rightarrow 
(-\infty, 0]$.
It follows that there exists a Morse function $\tilde{g}$,
such that
$\tilde{g}^{-1}(0) = \G$ and 
$\Hn(\tilde{g}^{-1}(-t)) \leq \Hn(\partial(f^{-1}([t,\infty))\setminus \Om))
+ \varepsilon/2$. We define $g(x) = - \tilde{g}(x)$ for
$x \in M \setminus \Om$ and extend it to a Morse function
from $M$ to $[-1, \infty)$ in an arbitrary way.
By property (2) of Lemma \ref*{morse_noncompact}
we have that (possibly after a small perturbation)
$\title{\Om}_1 = g^{-1}([-1,1)) \supset \Om_1$.

The bound on the area is similar to the 
argument in the proof of I.
It follows by (2)' that
$\Hn(\tilde{G}_t ) \leq \Hn(\G_t \setminus \Om) + \Hn(\G_t \cap \Om) + \varepsilon/2
< A+ \varepsilon$.
If $\Om \subset \Om_1$ then by property (4) of Lemma \ref*{morse_noncompact}
we may assume that $\tilde{\G}_1 = g^{-1}(1) = \G_1$.
\end{proof}

The second lemma in this section will deal with extending a Morse foliation.

The following result of Falconer
(\cite{Fa}, see also \cite[Appendix 6]{Gu1})
will be used in the proof.

\begin{theorem} \label{falconer}
(Falconer) There exists a constant $C(n)$ so that the following is true.
Let $U \subset \mathbb{R}^{n+1}$ be an open 
set with smooth boundary.
There exists a line $l \in \mathbb{R}^{n+1}$, so that
projection $p_l$ onto $l$ satisfies 
$Vol_{n}(U \cap p_l^{-1}(t)) < C(n) Vol_{n+1}(U)^{\frac{n}{n+1}}$
for all $t \in l$.
Moreover, we can assume that $p_l$ restricted to $\partial U$
is a Morse function.
\end{theorem}

\begin{lemma} \label{lem:extension}
Let $\varepsilon>0$, $L>0$.
Suppose $\Om_0 \subset \Om_1$ are bounded 
open sets with smooth boundary
and $\Om_1 \setminus \Om_0 \subset U$,
where $U$ is $(1+L)$-bilipschitz 
diffeomorphic to an open subset of $\mathbb{R}^{n+1}$.
There exists a constant $C(n)$ and a nested family $\{\G_t' \}$
with a family of corresponding open sets 
$\{\Om_t' \}$, such that 

(1) $\Hn(\G_t') \leq \Hn(\partial \Om_0)
+\Hn(\partial (\Om_1 \setminus \Om_0))+ C(n)(1+L)^n \Hnn(\Om_1 \setminus \Om_0)^{\frac{n}{n+1}}
+ \varepsilon$;

(2) $\Om_0'$ is an inward $\varepsilon$-perturbation
of $\Om_0$ and $\Om_1'=\Om_1$;

Alternatively, we can require that instead of (2)
the family satisfies 

(2') $\Om_1'$ is an outward $\varepsilon$-perturbation
of $\Om_1$ and $\Om_0'=\Om_0$;
\end{lemma}

\begin{proof}

Let $\Om'$ be an inward $\varepsilon/8$-perturbation 
of $\Om_1 \setminus \Om_0$.
By Theorem \ref*{falconer}
there exists a Morse function $f: \Om' \rightarrow [0,1]$
with fibers of area at most 
$C(n)(1+L)^n \Hnn(\Om_1 \setminus \Om_0)^{\frac{n}{n+1}}
+ \varepsilon/4$.
By Lemma \ref*{morse} there exists a nested sweepout of $\Om'$ 
$\{ \Si_t^a \}$ with a corresponding family of open sets
$\{ \Xi_t^a \}$, such that $\Xi_1^a = \Om'$ and
$\Hn(\Si_t^a) \leq \Hn(\partial (\Om_1 \setminus \Om_0))+ C(n)(1+L)^n \Hnn(\Om_1 \setminus \Om_0)^{\frac{n}{n+1}}
+ \varepsilon/2$.

Let $\{ \Si_t^b \}$ be a nested family with a corresponding family of open sets
$\{ \Xi_t^b \}$, such that $\Xi_0^b$ is an inward $\varepsilon/2$-perturbation
of $\Om_0$, $\Xi_1^b$ is an inward $\varepsilon/8$-perturbation
of $\Om_0$ and the areas of all hypersurfaces are at most
$\Hn(\partial \Om_0)+ \varepsilon/2$.

By Lemma \ref*{gluing_separated} there exists a nested family
$\{ \Si_t^c \}$ with a corresponding family of open sets
$\{ \Xi_t^c \}$, such that $\Xi_1^c = \Om_1$,
$\Xi_0^c = \Xi^1 \sqcup \Xi^2$,
where $\Xi^1$ is an inward $\varepsilon/8$-perturbation 
of $\Om_0$ and $\Xi^2$ is an inward $\varepsilon/8$-perturbation 
of $\Om_1 \setminus \Om_0$. It follows from the properties of perturbations
that, without any loss of generality, we may assume 
$\Xi^1 = \Xi_1^b$ and $\Xi^2 = \Om'$.

We define $\G_t' = \Si_{2t}^a \cup \Si_{2t}^b$
for $t \in [0,1/2)$ and $\G_t' = \Si_{2t-1}^c$
$t \in [0,1/2]$ with the open sets defined correspondingly.

We leave it to the reader to verify that
a similar construction yields a family satisfying
(2') instead of (2).
\end{proof}

\section{Nested sweepouts}


In this section we prove the following proposition.

\begin{proposition}
	\label{prop:nested}
For every $\varepsilon > 0$, given a family of hypersurfaces
$\{ \Gamma_t \}$ with the corresponding family
of open sets $\{ \Om_t \}$ and
$\Hn(\Gamma_t)\leq A$, there exists a nested family
$\{ \tilde{\Gamma}_t \}$  with the corresponding family
of open sets $\{ \tilde{\Om}_t \}$,
such that $\tilde{\Om}_0 \subset \Om_0$,
$\Om_1 \subset \tilde{\Om}_1$ and
$\Hn(\tilde{\Gamma}_t)\leq A + \varepsilon$.

In particular,	for any bounded open set
$U \subset M$ with smooth boundary we have $W(U) = W_n(U)$.

\end{proposition}



The proof proceeds in three steps.

\subsection{Step 1. Preliminary modification of the family.}
We start by replacing the original family $\{\G_t\}$
with a new family $\{\G_t'\}$
that possesses the property that
every hypersurface in the family nearly coincides in the complement
of a small ball with some hypersurface
from a finite list $\{\G_{t_i}' \}$.
This construction is inspired by
constructions of families, which are continuous
in the mass norm in the work of Pitts and 
Marques-Neves (see \cite[4.5]{Pi} and \cite[Theorem 14.1]{MN1}).

\begin{lemma} \label{continuity in mass norm}
 For any $\varepsilon>0$ there exists 
 a partition $0=t_0< ... < t_N=1$ of $[0,1]$ and
 a family $\{\G'_t\}$ with the corresponding
 family of open sets $\{\Om'_t\}$, such that 
 the following holds:
 
  (1.1) $\Om'_0 \subset \Om_0$ and $\Om_1 \subset \Om'_1$;
  
  (1.2) $\sup \{\Hn(\G'_t) \} < \sup \{\Hn(\G_t) \} + \varepsilon$;
  
  
  (1.3) For each $i=0,...,N-1$ we have one of the two possibilities:
  
        A. $\Om'_{t_i} \subset \Om'_{t_{i+1}}$ and there exists a Morse function
        $g_i: cl(\Om_{t_{i+1}} \setminus \Om_{t_i}) \rightarrow [t_i, t_{i+1}]$, such that
        $\G'_t = g_i^{-1}(t)$ and $\Om'_t = \Om_{t_i}' \cup g_i^{-1}(-\infty,t)$
        for $t \in [t_i, t_{i+1}]$.
  
        B. $\Om'_{t_{i+1}} \subset \Om'_{t_i}$ and there exists a Morse function
        $g_i: cl(\Om_{t_i} \setminus \Om_{t_{i+1}}) \rightarrow [t_i, t_{i+1}]$, such that
        $\G'_t = g_i^{-1}(t)$ and $\Om'_t = \Om_{t_i}' \setminus g_i^{-1}(-\infty,t]$
        for $t \in [t_i, t_{i+1}]$.

\end{lemma}

\begin{proof}
Let $M'$ be a compact subset of $M$
that contains the closure of $\Om_t$ for all $t \in [0,1]$.
Choose $r$ sufficiently small 
so that for 
 every ball $B$ of radius less than or equal to $r$ in $M'$ the following holds:

(i) $B$
is $(1+\frac{\varepsilon}{100W})^{1/n}-$bilipschitz diffeomorphic
to the Euclidean ball of the same radius;

(ii) $\Hn(B \cap \G_t) < \frac{\varepsilon}{20}$ for $t \in [0,1]$.

Condition (ii) can be realized because of the no concentration of 
mass property (s4).

Let $\{B_i\}$ be a collection of $k$ balls of radius $r$ covering
$M'$, such that balls of half the radius cover $M'$. We choose a partition 
$0=s_0< ... < s_{N'}=1$, such that 


(iii) $\Hnn(B_i \cap (\Om_{s_j} \setminus \Om_{s_{j+1}})) + 
\Hnn(B_i \cap (\Om_{s_{j+1}} \setminus \Om_{s_j})) < 
\min \{ \frac{r \varepsilon}{10k}, (\frac{\varepsilon}{10}) ^{\frac{n+1}{n}} \}$;

(iv) $|\Hn(B_i  \cap \G_{s_j}) - \Hn(B_i \cap \G_{s_{j+1}}| \leq \frac{\varepsilon}{10k}$

for each $j=0,...,N'$ and $i=1,...,k$.

We define the new family $\{\G'_t\}$ as follows.
For $t=s_j$ we set $\Om'_t = \Om_t$
and $\G'_t = \partial \Om_t$, unless $\G_t$ is a finite collection
of points in which case we set $\G'_t = \G_t$ and $\Om'_t = \emptyset$.

Define a subdivision of $[s_j, s_{j+1}]$ into $2k$ subintervals,
$s_j=s_j^0<...<s_j^{2k}=s_{j+1}$.
Let $\{B'_i\}$ be a collection of $k$ balls concentric with
$B_i$ of radius between $r/2$ and $r$ and such that
$\partial B'_i$ intersects $\G_{s_j}$ and $\G_{s_{j+1}}$
transversally. Set $U_j^1 = \Om_{s_j} \setminus \Om_{s_{j+1}}$
and $U_j^2 = \Om_{s_{j+1}} \setminus \Om_{s_j}$.
By coarea formula and property (iii)
for our choice of the subdivision $0=s_0< ... < s_{N'}=1$
we may assume that $B'_i$ satisfies
$\Hn(\partial B'_i \cap (U_j^1 \cup U_j^2)) \leq \frac{\varepsilon}{4k}$. 

By our choice of $B_i$ we have
that the collection of balls $\{B'_i\}_{i=1}^k$ still cover $M'$.
Inductively we define 
$$\Om'_{s_j ^{0}} = \Om'_{s_j}$$
$$\Om'_{s_j ^{2i-1}} = \Om'_{s_j ^{2i-2}} \setminus (B_i' \cap U_j^1)$$
$$\Om'_{s_j ^{2i}} = \Om'_{s_j ^{2i-1}} \cup (B_i' \cap U_j^2)$$ 
for $i=1,...,k$.

Surfaces $\partial \Om'_{s_j^{l}}$ may not be smooth,
but there exists an arbitrarily small perturbation so that the boundaries are smooth
(see Section \ref*{smoothing}).
We perform these perturbations in the inward direction for $\Om'_{s_j ^{2i-1}}$
and in the outward direction for $\Om'_{s_j ^{2i}}$. 
To simplify notation we do not rename the sets after the perturbations;
since the perturbations are arbitrarily small all the estimates 
for areas and volumes remain valid.

The following properties follow from the definition and (i)-(ii):


(a) $|\Hn(\partial \Om'_{s_j^{l}}) - \Hn(\G_{s_j})|< \varepsilon/2$;

(b) $\Om'_{s_j ^{2i-1}} \subset \Om_{s_j ^{2i}}$ and
$\Om_{s_j ^{2i-1}} \subset \Om_{s_j ^{2i-2}}$.


We define $\G'_{s_j^{l}} = \partial \Om'_{s_j^{l}}$, unless
$\Om'_{s_j^{l}}$ is empty. If $\Om'_{s_j^{l}}$ is empty we
set $\G'_{s_j^{l}}$ to be a point inside $\Om'_{s_j^{l-1}}$.
By properties (iii) and (iv) surface $\G'_{s_j^{l}}$
will satisfy the desired upper bound on the area.

To complete our construction we need to show 
existence of two types of nested families:
a nested family that starts on $\G'_{s_j^{2i-1}}$
and ends on $\G'_{s_j^{2i-2}}$;
a nested family that starts on $\G'_{s_j^{2i-1}}$
and ends on $\G'_{s_j^{2i}}$.
In both cases we want the homotopies to 
satisfy the desired upper bound on the areas.

Consider the set $\Om'_{s_j^{2i-2}} \setminus \Om'_{s_j^{2i-1}} 
= B_i \cap U_j^1$.
After smoothing the corner (see Section \ref*{smoothing})
we call this set $U$.
We map $B_{i}$ to $\mathbb{R}^{n+1}$ by 
a $(1+\frac{\varepsilon}{100W})^{1/n}$-bilipschitz 
diffeomorphism. 
Existence of the desired nested families
follows by Lemma \ref*{lem:extension}.
The upper bound for the area follows form the upper bound
for the area at the endpoints and properties (i) and (ii).
%
\end{proof}

\subsection{Step 2. Local monotonization} 
Assume that family $\{\G_t\}$ satisfies
conclusions of Lemma \ref*{continuity in mass norm}
for the subdivision $0=t_0< ... < t_N=1$.

For every $\varepsilon>0$ and each $i=0,...,N-1$ we will define sets
$\Om^i_0$ and $\Om^i_1$,
such that the following holds:

(2.1) $\Om^i_0 \subset \Om^i_1$;

(2.2) $\max \{ \Hn( \partial \Om^i_0), 
\Hn(\partial \Om^i_1) \} 
    \leq \max \{ \Hn(\G_{t_i}), \Hn(\G_{t_{i+1}})\} + \varepsilon$;

(2.3) $\Om_{t_{i+1}} \subset \Om^i_1$ and
$  \Om^i_0 \subset \Om_{t_{i}}$;
        

(2.4) There exists a nested family of hypersurfaces
$\{\G^i_t\}$, $0 \leq t \leq 1$,
with the corresponding 
family of nested open sets
$\Om^i_t$, such that
$\Hn(\G^i_t) \leq \max \{ \Hn(\G_{t_i}), \Hn(\G_{t_{i+1}})\} + \varepsilon$.

\textbf{Definition of $\Om^i_0$ and $\Om^i_1$}

Assume (2.1) - (2.4) are satisfied for all $\Om^j_0$ and $\Om^j_1$ for
$j<i$.
By Lemma \ref*{continuity in mass norm} (1.3)
we only need to consider the following two cases:

(A) $\Om_{t_{i}} \subset \Om_{t_{i+1}}$. In this first case we define $\Om^i_0 = \Om_{t_{i}}$
and $\Om^i_1 = \Om_{t_{i+1}}$. Properties (2.1)-(2.3) follow
immediately from the definition. 
Property (2.5) follows by Lemma \ref*{continuity in mass norm} (1.3).

(B) $\Om_{t_{i+1}} \subset \Om_{t_{i}}$. 
We define $\Om^i_0 = \Om_{t_{i+1}} \setminus cl(N_{\delta}( \partial \Om_{t_{i+1}}))$,
where $\delta>0$ is chosen sufficiently small so that $cl(N_{\delta}( \partial \Om_{t_{i+1}}))$
is diffeomorphic to $\partial \Om_{t_{i+1}} \times [-\delta, \delta]$ and hypersurfaces
equidistant from $\partial \Om_{t_{i+1}}$ in this neighbourhood all have areas
less than $\Hn (\partial \Om_{t_{i+1}}) + \varepsilon/2$. We set $\Om^i_1 = \Om_{t_{i+1}}$.


It is straightforward to verify that with these definitions
$\Om^i_0$ and $\Om^i_1$ satisfy (2.1)-(2.4).

The following important property is an immediate consequence of (2.3):

(2.5) $\Om^{i+1}_0 \subset \Om^{i}_1$.

Informally, the reason why (2.5) holds is because
to construct $\Om^{i+1}_0$ we push $\Om_{t_{i+1}}$ 
inwards (or not at all) and to construct $\Om^{i}_1$ we push
$\Om_{t_{i+1}}$ outwards (or not at all).

\subsection{Step 3. Gluing two nested families} 

We prove the following:

\begin{proposition} \label{gluing}
Suppose $\{\G_t^a\}$ and $\{\G_t^b\}$
are two nested families
(with corresponding families of open sets
 $\{\Om_t^a\}$ and $\{\Om_t^b\}$ respectively)
and $\Hn(\G_t^i) \leq W$. 
Suppose moreover that $\Om_0^b \subset \Om_1^a$.
For any $\varepsilon> 0$ there exists a nested family
$\{\G_t\}$ and a corresponding family 
of open sets $\{\Om_t\}$, such that 
$\Hn(\G_t) \leq W+ \varepsilon$,
$  \Omega_1 ^b \subset \Omega_1$
and $ \Omega_0 \subset \Omega_0 ^a$.
\end{proposition}



\begin{proof}

	The idea for the proof
	is shown in Figure \ref*{fig:split}.

	\begin{figure} 
		\centering
		\includegraphics[scale=0.35]{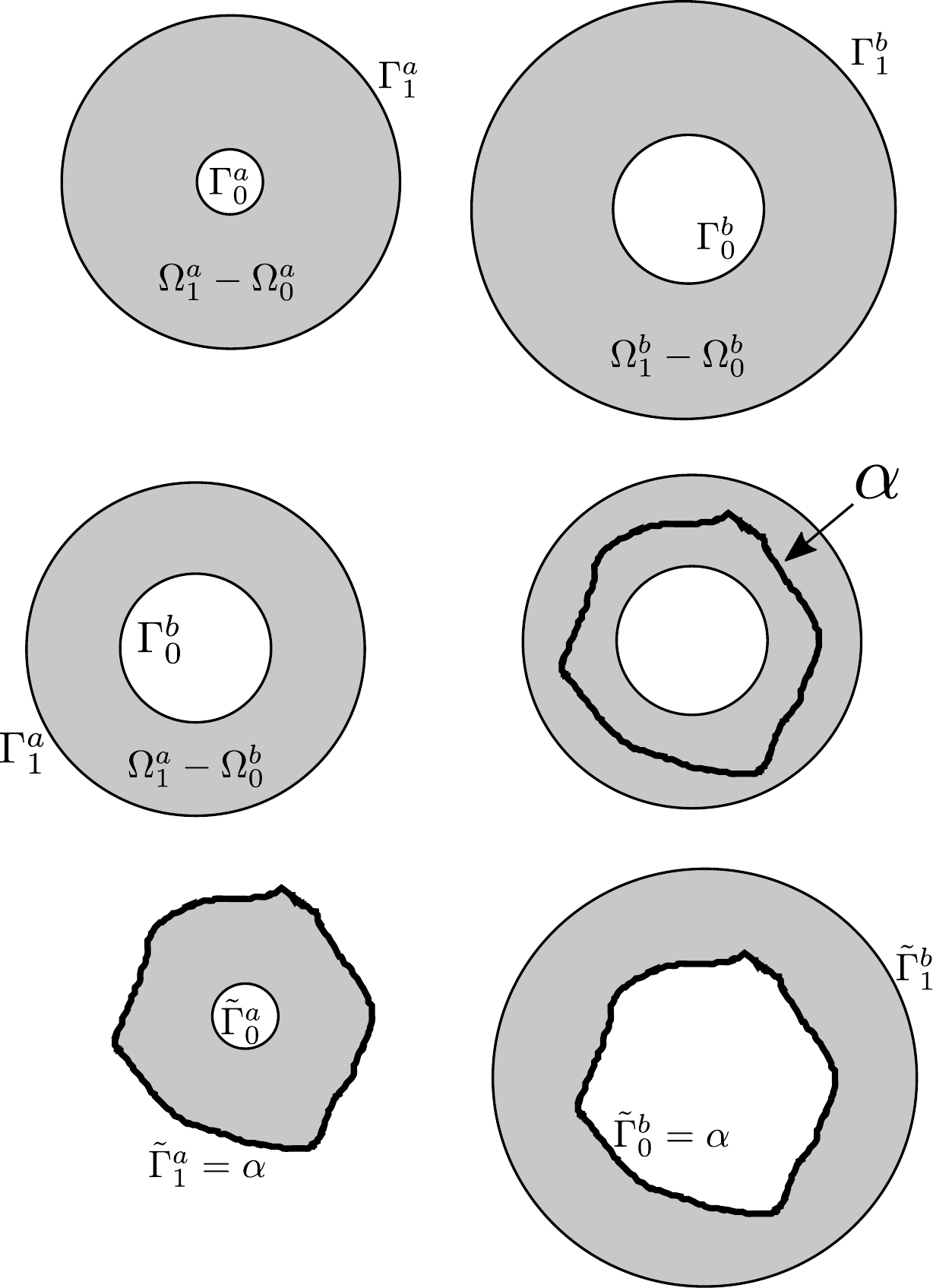}
		\caption{Gluing two nested sweepouts.}
		\label{fig:split}
	\end{figure}

Let $\mathcal{S}$ denote the collection of all
open sets $\Om'$, such that $\Om_0^b \subset \Om' \subset \Om_1 ^a$
and $\partial \Om'$ is smooth.
Let $A = \inf_{\Om' \in \mathcal{S}} \Hn(\partial \Om')$ and
choose $\Om \in \mathcal{S}$ with
and $\Hn(\partial \Om) < A + \varepsilon/4$. 
We set $\alpha = \partial \Om$.

We claim that $\Om$ and $\alpha$ satisfy properties (i) and (ii)
from Lemma \ref*{lem:split}(I) for $\Om_t = \Om_t^a$.
Indeed, if $\Om'$ satisfies $\Om \subset \Om' \subset  \Om_1 ^a$
then $\Om' \in \mathcal{S}$ and $\Hn(\partial \Om') <
\Hn(\alpha) + \varepsilon/4$.
By Lemma \ref*{lem:split}(I) there exists
a nested family $\{\tilde{\G}_t^a\}$
with the corresponding family of open sets $\{\tilde{\Om}_t^a\}$,
such that $\tilde{\Om}_0^a \subset \Om_0^a$, $\tilde{\G}_1^a = \alpha$ and
$\Hn(\tilde{\G}_t^a) \leq W+\varepsilon$.

We claim that $\Om$ and $\alpha$ also satisfy properties (i)' and (ii)'
from Lemma \ref*{lem:split}(II) for $\Om_t = \Om_t^b$.
Indeed, if there is an open set $\Om'$ with
$\Om_0^b \subset \Om' \subset \Om$ then again we have
$\Om' \in \mathcal{S}$ and inequality 
$\Hn(\partial \Om') <
\Hn(\alpha) + \varepsilon/4$ follows by definition of $\Om$.
By Lemma \ref*{lem:split}(II) there exists
a nested family $\{\tilde{\G}_t^b\}$
with the corresponding family of open sets $\{\tilde{\Om}_t^b\}$,
such that  $\Om_1^b \subset \tilde{\Om}_1^b$, $\tilde{\G}_0^b = \alpha$ and
$\Hn(\tilde{\G}_t^b) \leq W+\varepsilon$.

We define the desired nested family $\G_t$ 
simply by concatenating these two nested families.

\end{proof}

Now we are ready to complete the proof of
Proposition \ref*{prop:nested}.
We apply local monotonization
to define families $\{\G^i_t\}$ for $i=1,...,N-1$.

By (2.5) we have $\Om_0^2 \subset \Om_1^1$.
Hence, we can apply  Proposition \ref*{gluing}
to the nested families $\{\G_t^1\}$ and $\{\G_t^2\}$.
We obtain a new nested family $\G_t^{1,2}$
with the corresponding family of open sets
$\{ \Om^{1,2}_t \}$. By (2.3) and Proposition \ref*{gluing}
we have $\Om^{1,2}_0 \subset \Om_0^1 \subset \Om_0$ and
$\Om_{t_2} \subset \Om^2_1 \subset \Om^{1,2}_1$.
Using (2.5) again we have
$\Om^3_0 \subset \Om^{1,2}_1$. Hence, we can apply Proposition
\ref*{gluing} to $\{ \G^{1,2}_t \}$ and $\{ \G^3_t \}$.
We iterate this procedure. At the $i$-th step we
apply Proposition \ref*{gluing} to families
$\{ \G^{1,...,i}_t \}$ and $\{ \G^{i+1}_t \}$ to construct a new
nested family $\{ \G^{1,...,i,i+1}_t \}$ with
$\Om^{1,...,i}_0 \subset \Om_0$ and 
$\Om_1 \subset \Om^{1,...,i}_1$. 
Proposition \ref*{gluing} and (2.5) guarantee that 
$\Om^{i+2}_0 \subset \Om^{1,...,i}_1$, so we can go 
to the next step.

After performing this operation $N$ times we obtain 
the desired nested family. This finishes the 
proof of Theorem \ref*{prop:nested}.

\section{No escape to infinity} \label{sec: no escape}


In this section we prove Proposition \ref*{prop:intersection_U}, which we recall below.	

{\bf Proposition 2.1}
\textit{
For every good set $U$ there exists a positive constant $\varepsilon(U)$ which depends only on $U$ such that the following holds.
For every good sweepout $\{ \Gamma_t \}$ of $U$ with associated family of open sets $\{ \Om_t \}$, 
there is a surface $\Gamma_{t'}$ in the collection
which has area at least $W_g(U)$, and such that $\Hn( \Gamma_{t'} \cap cl(U) ) \geq \varepsilon(U)$.
}


The proof is by contradiction. 
We assume that Proposition \ref*{prop:intersection_U}
does not hold and construct a good sweepout
with volume of hypersurfaces strictly less
than $W_g(U)$. 
The main tool in the proof is Theorem 
\ref*{prop:nested}. 

%
%

Let $U$ be a good set. 

\begin{lemma} \label{epsilon}
There exit $\varepsilon(U)> 0$,
$\varepsilon_0(U)>0$ and $\varepsilon_1(U)>0$ such that 
for any open set $\Om'$
the following
holds:

    (1) $\max\{ \varepsilon, \varepsilon_1\} < 
    \Hn(\partial U)/10$.

    (2) If $\varepsilon_0 <\Hnn(\Om' \cap U)< 
    \Hnn(U) - \varepsilon_0$ then 
    $\Hn(\partial \Om' \cap U) > 2 \varepsilon$.
    
    (3) A) If $\Hnn(\Om' \cap U) < 2 \varepsilon_0$
    then there exists a family of open sets $\{ \Xi_t \}$
    with $\Xi_0 = \Om'$,
    $\Xi_t \setminus N_{\varepsilon_1} (U) = \Om' \setminus N_{\varepsilon_1} (U)$, $\Xi_1 \cap U = \emptyset$ and
    $\Hn(\partial \Xi_t)<  \Hn(\partial \Om')+ \Hn(\partial U) + \varepsilon_1$.
    
    B) If $\Hnn(\Om' \cap U) > \Hnn(U) - 2 \varepsilon_0$
    then there exists a family of open sets $\{ \Xi_t \}$
    with $\Xi_0 = \Om'$,
    $\Xi_t \setminus N_{\varepsilon_1} (U) = \Om' \setminus N_{\varepsilon_1} (U)$, $\Xi_1 \cap U=U$ and
    $\Hn(\partial \Xi_t)<  \Hn(\partial \Om')+ \Hn(\partial U)+ \varepsilon_1$.

\end{lemma}

\begin{proof}
Pick any $\varepsilon_1 \in (0,\Hn(\partial U)/10)$.
We will show that for all sufficiently small $\varepsilon_0$
(with the choice of $\varepsilon_0$ depending on $\varepsilon_1$)
statement (3) holds; we will show that
for all sufficiently small 
$\varepsilon$ 
(with the choice of $\varepsilon$ depending on $\varepsilon_0$)
statement (2) holds.

Statement (2) follows from the properties of the isoperimetric profile
of $cl(U)$.

Now we will prove Statement (3) A).
Statement (3) B) follows by an analogous argument.

Let $r_0>0$ be sufficiently small, so that every 
ball $B$ of radius $r \in (0,r_0]$
centered at a point in $cl(U)$
is $2$-bilipschitz diffeomorphic
to a ball of the same radius in the Euclidean space.

Choose a covering $\{ B_i \}$ of $cl(U)$ by $N$ balls of radius $r_0$,
so that concentric balls of radius $\frac{r_0}{2}$,
denoted by $\frac{1}{2}B_i$, still cover $cl(U)$.
Set $\varepsilon_0 = \min\{ \frac{\varepsilon_1 r_0}{20 N}, 
( \frac{\varepsilon_1}{2^{n+2} C(n)})^{\frac{n+1}{n}}\}$,
where $C(n)$ is the constant from Lemma \ref*{lem:extension}.
Using coarea inequality we may choose a covering
$\{ B_i' \}$ of $U$ by $N$ balls of radius $r_i \in (r_0/2,r_0)$,
so that $\Hn((\partial B_i') \cap (\Om' \cap U)) \leq \frac{4 \varepsilon_0}{r_0}
\leq \frac{\varepsilon_1}{5 N}$.



We inductively push $\Om'$ outside of $B_1' \cap U, B_2' \cap U, ...,
B_N' \cap U$.  

By Lemma \ref*{lem:extension} there exists 
a nested family that starts on $\Om'$ and
ends on a smoothing of $\Om' \setminus (B_1' \cap U)$.
We can choose the smoothing so that the set does not intersect
$B_1' \cap U$.
In the process we have increased the area by 
at most $\Hn(B_1' \cap \partial U) + 
\Hn(\partial B_1'\cap (U \cap \Om')) + 2^n C(n) \varepsilon_0^{\frac{n}{n+1}}$.
By our choice of $\varepsilon_0$ we conclude that 
the area increased by at most $\Hn(B_1' \cap \partial U) + \frac{\varepsilon_1}{2N}$.



We iterate this procedure for each ball $B_i'$.
During the $i$-th step, $1 \leq i \leq N$, we have that 
the area of the hypersurface is bounded by
$\Hn(\partial \Om')+ \Hn(\partial U) + \frac{i \varepsilon_1}{2N}$.
This concludes the proof of  Statement (3) A).
\end{proof}


	


%

\emph{Proof of Proposition \ref*{prop:intersection_U}}.
Suppose Proposition \ref*{prop:intersection_U} does not
hold. 
Then there exists
a good sweepout $\{ \G_t \}_{t \in [0,1]}$,
such that
if $\Hn(\G_t) \geq W_g(U)$ then
$\Hn(\G_t \cap U) < \varepsilon(U)$.
Let $\{\Om_t \}$ denote the corresponding
family of open sets. 
Let $f(t) = \Hn(\G_t \cap U)$. Note that $f(t)$
may not be continuous. However, it is easy to see
that one can perturb the family $\{ \G_t \}$
so that it is roughly continuous in the following sense.

\begin{definition}
Function $f(t)$ is $\delta$-continuous if
the oscillation 

\noindent
$\omega_f(t) = \lim_{a \rightarrow 0}[ \sup_{s \in [t-a,t+a]} f(s) -
\inf_{s \in [t-a,t+a]} f(s)]$ satisfies $\omega_f(t) < \delta$ for every
$t$.
\end{definition}

\begin{lemma} \label{area_continuity}
Let $U$ be a bounded open set with smooth boundary
and $\{ \G_t \}$ be a good sweepout of $U$.
For every $\delta>0$ there exists a 
good sweepout
$\{ \G'_t \}$ of $U$, such that
$f(t) = \Hn(\G'_t \cap U)$ is $\delta$-continuous,
$\sup_t \Hn(\G'_t) \leq \sup_t \Hn(\G_t) + \delta$
and $\sup_t \Hn(\G'_t \cap U) \leq \sup_t \Hn(\G_t \cap U) + \delta$.
\end{lemma}

\begin{proof}
This follows from the construction in the proof
of Lemma \ref*{continuity in mass norm}.
\end{proof}

Hence, without any loss of generality we may assume that 
sweepout $\{ \G_t \}$ satisfies the conclusions of Lemma \ref*{area_continuity}
for $\delta < \varepsilon/10$ and that 
for all $\G_t$ with $\Hn(\G_t) \geq W_g(U)$ we have
$\Hn(\G_t \cap U) < 1.1 \varepsilon(U)$.

Let $g: [0,1] \rightarrow
[0, \Hnn(U)]$ be defined as
$g(t)= \Hnn(U \cap \Om_t)$.
Function $g(t)$ is continuous.
By Lemma \ref*{epsilon} (2) each connected component $I'$ of
$g^{-1}([\varepsilon_0, \Hnn(U) - \varepsilon_0])$
is contained in some interval $I=[t_0,t_1] \subset [0,1]$, such that
$f(t) \geq \frac{3}{2} \varepsilon$ for all $t \in I$.
Moreover, by Lemma \ref*{area_continuity}
we may assume that $\varepsilon \leq f(t_i) \leq 2 \varepsilon$, $i=0,1$.
By continuity of $g(t)$ and since $\{ \G_t \}$
is a sweepout there exists an interval $I$ as above with
$\Hnn(\Om_{t_0} \cap U) \leq \varepsilon_0$
and $\Hnn(\Om_{t_1} \cap U) \geq \Hnn(U)- \varepsilon_0$.




By construction
 we have that
$\Hn(\G_t) < W_g(U) - \delta$ for some $\delta>0$ and
for all $t \in I$.

We would like to turn $\{\G_t\}$ into
a good sweepout of $U$, while retaining 
an upper bound on the volume below $W_g(U)$.
The family $\{\G_t\}_{t \in I}$ fails 
to be a good sweepout of $U$ for two reasons:

1. $\Om_{t_0} \cap U$ 
and $\Om_{t_1} \setminus U$ are not empty; 

2. $\Hn(\G_{t_0})$ and $\Hn(\G_{t_1})$
may be larger than $\Cgs$. In fact, they
may be as large as the largest hypersurface
in $\{\G_t\}_{t \in I}$.

To address the first problem we note
that $\Om_{t_0} \cap U$ 
and $\Om_{t_1} \setminus U$
have volume at most $\varepsilon_0$
and we may use Lemma \ref*{epsilon}
to homotope $\G_{t_0}$ and $\G_{t_1}$
outside of $U$ while increasing the
$\Hn-$measure of the hypersurfaces
by a controlled amount. Observe, however,
that if $\delta$ is much smaller than $\varepsilon$
and $\Hn(\G_{t_i})$ is almost equal to 
$W_g(U) - \delta$ then the resulting family
will have volume larger than $W_g(U)$.
The second problem seems even more substantial.

The main tool to resolve these two problems
is to replace $\{\G_t\}_{t \in I}$ with 
a nested family. This allows
us to define certain two nearly area 
minimizing hypersurfaces. We then 
modify the nested family so that it starts and
ends on these two hypersurfaces, 
which have small area and can be ``homotoped"
away from $U$ to produce a good sweepout.

We apply Proposition \ref*{prop:nested}
to construct a nested family $\{\bar{\G}_t \}$, $t \in [0,1]$,
such that $\Hn(\bar{\G}_t) < W_g(U) - \frac{\delta}{2}$,
$ \bar{\Om}_0 \subset \Om_{t_0}$ and
$\Om_{t_1} \subset \bar{\Om}_1$.

\begin{figure} 
   \centering
\includegraphics[scale=0.5]{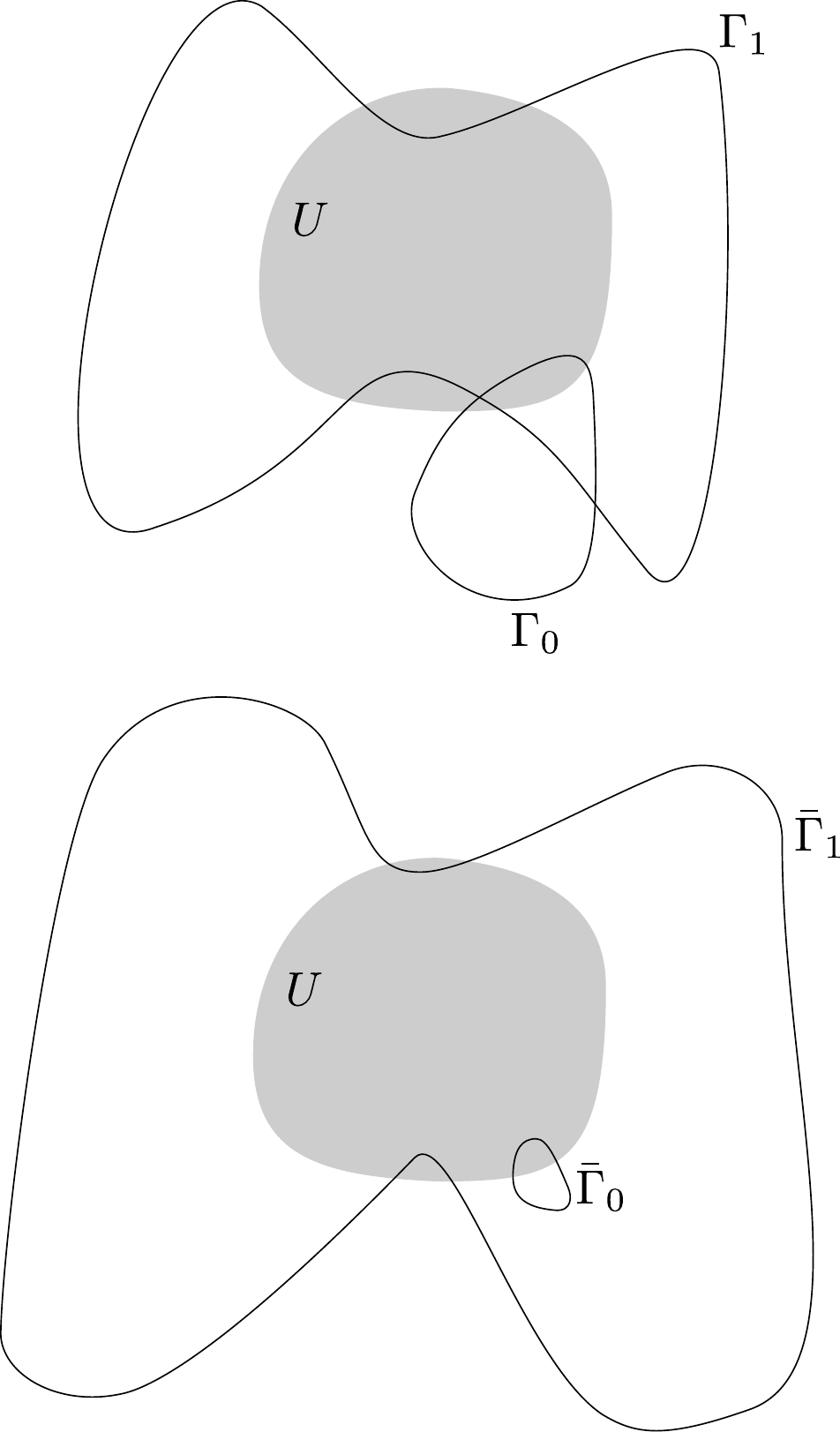}
\caption{Replacing family $\{\G_t\}_{t \in I}$ with a nested family $\{\bar{\G}_t \}$}
\label{replacement}
\end{figure}

\begin{figure} 
   \centering
\includegraphics[scale=0.5]{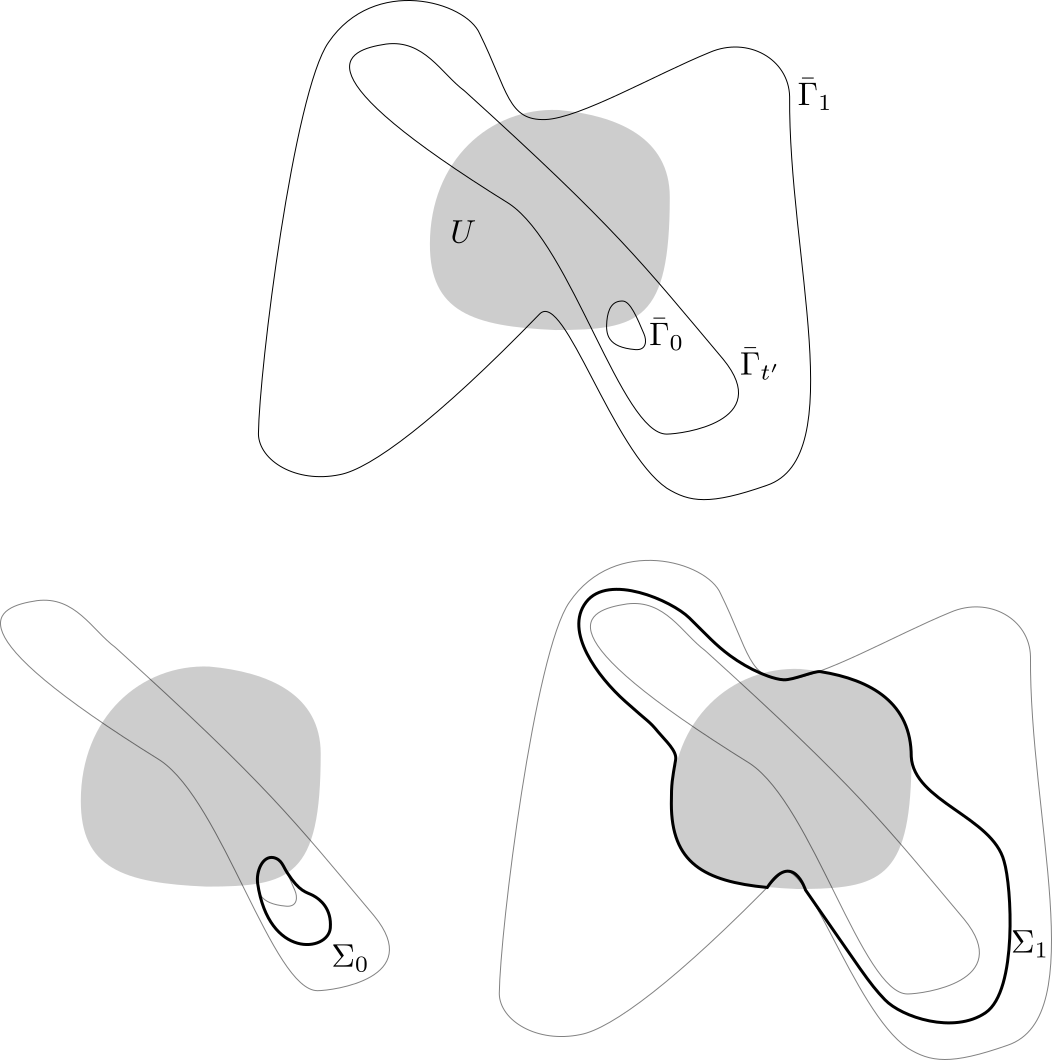}
\caption{Constructing a good sweepout in the proof of Proposition \ref*{prop:intersection_U}.}
\label{fig:final}
\end{figure}


The situation is depicted on Figure \ref*{replacement}.
Let $P = (\Om_{t_0} \cap U)
\cup (U \setminus cl(\Om_{t_1}))$. Let
$\bar{U}$ be an inward $\delta$-perturbation 
$U \setminus cl(P)$.

We see that $\bar{U}$ is contained in $U$
and up to a controllable error has the same volume and boundary area.
We summarize important properties of $\bar{U}$:

(i) $\{ \bar{\G}_t \}$ is a nested sweepout of
$\bar{U}$;

(ii) $\Hn(\partial \bar{U}) \leq \Hn(\partial U) + 4 \varepsilon + \delta \leq 2 \Hn(\partial U)$;

(iii) There exists a homotopy pushing out $\partial \bar{U}$
outside of $U$ through hypersurfaces of area at most $3 \Hn(\partial U)$.

Property (iii) follows from Lemma \ref*{epsilon} (B).


%
%

\begin{lemma} \label{long slice}
There exists $t' \in [0,1]$,
such that
$\Hn(\bar{\G}_{t'} \setminus \bar{U}) \leq 
2 \Hn( \partial U)$.
\end{lemma}

\begin{proof}

Let $L = \max_t \{\Hn(\bar{\G}_{t} \cap \bar{U}) \}$.
Since $\{ \bar{\G}_t \}$ is a nested sweepout of
$\bar{U}$
we can apply Lemma \ref*{morse} to obtain
nested sweepout of $\bar{U}$ by hypersurfaces
of area at most 

\begin{align*} 
L + \Hn(\partial \bar{U}) + \delta & \leq L + \Hn(\partial U) + 4 \varepsilon + 2 \delta \\
    & \leq L +  2 \Hn( \partial U)
\end{align*}

Moreover, this sweepout starts on a hypersurface of
area $0$ and ends on $\partial \bar{U}$.
By (iii) we can deform
$\partial \bar{U}$ outside of $U$ through hypersurfaces
of controlled area. 



We have produced a
good sweepout of $U$ with maximal volume of the hypersurface at most
$L +  2 \Hn( \partial U)$. 
By definition of $W_g(U)$ we have $L +  2 \Hn( \partial U) \leq W_g(U)$.
Hence, 
$\Hn(\bar{\G}_{t'}) < W_g(U)$
implies that for some $t' \in [0,1]$ we have $\Hn(\bar{\G}_{t'} \setminus \bar{U}) 
<2 \Hn( \partial U)$.

\end{proof}


We will construct a sweepout of
$\bar{U}$
with hypersurfaces of area at most
$W_{g}(U) - \delta$, starting and ending on hypersurfaces of area
less than $3 \Hn(\partial U)$.
By Lemma \ref*{epsilon} we can deform it into
a good sweepout of $U$
by hypersurfaces of area at most $W_{g}(U) - \delta/4$.
This contradicts the definition of $W_{g}(U)$ 
and so Proposition \ref*{prop:intersection_U} follows.

To construct a sweepout of $\bar{U}$
with these properties we proceed as follows.
Let $t'$ be as in Lemma \ref*{long slice}, and
let $\mathcal{U}_0$ denote the collection
of all open sets $\Om$ with smooth boundary,
such that $\bar{\Om}_0 \subset \Om \subset \bar{\Om}_{t'} \setminus \bar{U}$.
Let $\mathcal{U}_1$ denote the collection
of all open sets $\Om$ with smooth boundary,
such that $\bar{\Om}_{t'} \cup \bar{U} \subset \Om \subset \bar{\Om}_1$.
Let $A_i = \inf \{\Hn(\partial \Om): 
\Om \in \mathcal{U}_i\}$. 
Observe that a perturbation of $\bar{\Om}_{t'} \setminus cl(\bar{U})$
 is an element of $ \mathcal{U}_0$ and 
a perturbation of $\bar{\Om}_{t'} \cup \bar{U}$ is an element of
$\mathcal{U}_1$. By Lemma \ref*{long slice}
the boundary areas of these hypersurfaces are at most $3 \Hn(\partial U)$.
We conclude that
$A_i \leq 3 \Hn(\partial U)$.
Let $\Si_0 = \partial \Xi_0$  and $\Si_1 = \partial \Xi_1$
be two hypersurfaces with 
$\Xi_i \in \mathcal{U}_i$ and
$\Hn(\Si_i) \leq A_i + \delta/4$.  We have that $\Xi_0$ is contained in $\bar{\Om}_{t'}$,
and that $\bar{U}$ is contained in its complement, and we also have that $\Xi_1$
contains both $\bar{U}$ and $\bar{\Om}_{t'}$.  In particular, the set $\Xi_1 \setminus \Xi_0$
contains $\bar{U}$.

We apply Lemma \ref*{lem:split} I
to construct a nested sweepout of $\bar{U}$ that starts on $\Si_0$ and ends on $\bar{\Om}_1$ and is composed of
hypersurfaces of area at most $W_{g}(U) - 3 \delta/4$.  Here we are using the fact that $\Xi_0$ is contained in
$\bar{\Om}_1$.  We then apply Lemma \ref*{lem:split} II to this sweepout to produce a nested sweepout of $\bar{U}$
that starts on $\Si_0$ and ends on $\Si_1$ and is composed of hypersurfaces of area at most 
$W_{g}(U) - \delta/4$.  Here we are using the fact that $\Xi_0 \subset \Xi_1$.
This finishes the proof of Proposition \ref*{prop:intersection_U}.  This proof is shown
in Figure \ref*{fig:final}.

\section{Convergence of a min-max sequence to 
a minimal hypersurface}
\label{sec: convergence}

\subsection{Manifolds with sublinear volume growth}
In this section we prove Theorem \ref*{main} and Corollary \ref*{main'}.
Corollary \ref*{main'} follows from the following lemma.
We show that if $M$ has sublinear volume growth
(in particular, if it has finite volume) then it contains 
a good set.

\begin{lemma}
Let $M^{n+1}$ be a complete non-compact manifold
with sublinear volume growth.
There exists a good set $U \subset M$, such that
$0< W_g(U) < \infty$.
\end{lemma}

\begin{proof}
Let $x$ be such that $\liminf_{r \rightarrow \infty} \frac{Vol(B_r(x))}{r} =0 $
Fix a small geodesic ball $B_r(x)$
and define an isoperimetric constant $C_I = \inf \{\Hn(\Si) \}$,
where the infimum is taken over all
hypersurfaces in $B_r(x)$, subdividing $B_r(x)$
into two subsets of equal volume.
By the coarea formula we can find $R>r$ with
$\Hn(\partial B_R(x)) < \frac{C_I}{100}$
and $\partial B_R(x)$ smooth.

It follows that $B_R(x)$ is a good set.
The distance function $d_x(y) = dist(x,y)$
may not be smooth, but there exists
a smoothing of this function $\tilde{d}_x$ (see \cite{GW}),
such that $\tilde{d}_x = d_x$ in 
$B_R(x)$ and $|\nabla \tilde{d}_x| \leq 1 +\varepsilon$ for all $y$.
Moreover, we may assume that $\tilde{d}_x$ is a Morse function.

Hence, 
the set of good sweepouts of $B_R(x)$ is non-empty.
Every sweepout of $B_R(x)$ is also a sweepout of 
$B_r(x)$, so it must contain a hypersurface of area
at least $C_I$.
\end{proof}

\subsection{Proof of Theorem \ref*{main}}
Theorem \ref*{main} follows immediately from 
the following Theorem.

\begin{theorem} \label{main_full}
Let $M^{n+1}$ be a complete Riemannian
manifold of dimension $n+1$.
Suppose $M$ contains a good set $U$.
For every $\delta>0$ there exists a 
complete embedded minimal hypersurface $\G$,
satisfying the following properties:
\begin{enumerate}
    \item $\Hn(\G) \leq W_{\partial}(U) + \Hn(\partial U)$;
    \item $\Hn(\G \cap N_{\delta}(U)) \geq \frac{\varepsilon(U)}{2}$,
\end{enumerate}
where $\varepsilon(U)$ is as in Lemma \ref*{epsilon}.
The hypersurface is smooth
in the complement of a closed set of dimension $n-7$.
\end{theorem}

\begin{remark} \label{main_remark}
a) It seems that the min-max argument applied
to families, which are good sweepouts of a good set $U$
may produce a non-compact minimal hypersurface.
Consider the following heuristic example.
Let $S_r$ denote spheres of radius $r$ in $\mathbb{R}^3$.
We modify the Euclidean metric on $\mathbb{R}^3$,
so that the new metric is invariant under rotations around $0$,
and so that the areas of $S_r$ and lengths of great circles on $S_r$ decay exponentially
for $r>1$. If the decay is fast enough the min-max argument for
good sweepouts of the ball $B_2(0)$ seems to produce 
a hyperplane passing through $0$ (of area $\pi + \varepsilon$).

b) If $U$ admits a
metric of non-negative Ricci curvature in the same conformal class
then from \cite{GL} we obtain an
upper bound for the volume of the minimal hypersurface
$\Hn(\G)\leq C(n) \Hnn(U)^{\frac{n}{n+1}}$.

c) The theorem does not exclude the possibility
that the volume of $\G$ is much smaller than $W_g(U)$.
As the min-max sequence converges (weakly) to $\G$ some 
non-zero mass may escape into the ends. 
Nonetheless, we are still able to construct
an almost minimizing min-max sequence converging
to a stationary varifold, so that
regularity arguments of \cite{Pi} apply.
\end{remark}

To prove Theorem \ref*{main_full} we use Proposition \ref*{prop:intersection_U} 
and arguments from \cite{DT}.
For the most part in this section we closely follow \cite{DT}.
However, some modifications are necessary 
in construction of the pull-tight deformation and 
construction of a min-max sequence, which is almost minimizing
in all sufficiently small annuli.

The regularity of a stationary varifold obtained
from a min-max sequence is proved using the notion 
of $\varepsilon$-almost minimizing hypersurfaces introduced
in \cite{Pi}. We will use the notion of almost minimality from 
\cite[2.2]{DT}. 

\begin{definition}
Let $\varepsilon>0$  and $U \subset M$ open. A boundary $\partial \Om$
is called $\varepsilon$-almost minimizing in U if there is NO 1-parameter
family of boundaries $\{ \partial \Om_t \}$, $t \in [0, 1]$, 
satisfying the following properties:

\begin{itemize}
    \item (s1), (s2), (s3), (s4), (sw1), and (sw3) of Definition \ref*{def: sweepout} hold;
    \item $\partial \Om_0 = \Om$ and $\partial \Om _t  \setminus U = \partial \Om  \setminus U$
    for every $t$;
    \item $\Hn(\partial \Om_t) \leq \Hn(\partial \Om) + \frac{1}{8} \varepsilon$;
    \item $\Hn(\partial \Om_1) \leq \Hn(\partial \Om) - \varepsilon$
\end{itemize}

A sequence $\{\partial \Om_k \}$ of hypersurfaces is called almost minimizing in $U$ if
each $\partial \Om_k$ is $\varepsilon_k$-almost minimizing in $U$ 
for some sequence $\varepsilon_k \rightarrow 0$.
\end{definition}

Let $\mathcal{AN}_{r}(x)$ denote the 
set of all open annuli $An(x,t_1,t_2)= B_{t_2}(x) \setminus cl( B_{t_2}(x))$
for $t_1 < t_2 <r$.
We have the following result from \cite{DT}:

\begin{proposition} \label{regularity}
Let $r:M \rightarrow \mathbb{R}_+$ be a function and
$\{ \G^k \}$ is a sequence of
hypersurfaces, s.t.

(A) $\{ \G^k \}$ is a.m. in every $An(x) \in \mathcal{AN}_{r(x)}(x)$;

(B) $\G^k$ converges to a stationary varifold $V$
as $k \rightarrow \infty$.

Then $V$ is induced by an embedded minimal hypersurface,
which is smooth on the complement of a closed set of Hausdorff dimension
at most $n-7$.
\end{proposition}

\begin{proof}
This proposition is contained in Propositions 2.6, 2.7 and 2.8 of 
\cite{DT}. All arguments there are local and therefore they apply
to the non-compact case. 
\end{proof}

\begin{proposition} \label{sequence existence}
Let $U \subset M$ be a good set and suppose $W_g(U)< \infty$.
For every $\delta>0$ there exists a function $r:M \rightarrow \mathbb{R}_+$,
$\varepsilon>0$
and a sequence $\{ \G^k \}$,
such that  (A) and (B) of Proposition \ref*{regularity}
hold and

(C) $\Hn(\G^k \cap N_{\delta}(U)) > \varepsilon/2$ for every $k$.

\end{proposition}


Combining Propositions \ref*{regularity} and \ref*{sequence existence}
we obtain that $M$ contains a stationary varifold $V$ induced by a minimal hypersurface $\Si$
with $\Hn(\Si \cap N_{\delta}(U)) > \varepsilon/2$. In particular, the intersection 
of $\Si$ with $N_{\delta}(U)$ is non-empty and the minimal hypersurface has volume at
least $\varepsilon/2$.
This implies Theorem \ref*{main_full}.

The rest of this section will be devoted to the proof of Proposition \ref*{sequence existence}.

\subsection{Pull-tight}


Using terminology from \cite{DT}
we say that a sequence $\{ \G_t^i \}$
of good sweepouts
of $U$ is minimizing if 
$lim_{i \rightarrow \infty} 
sup_t \Hn(\G_t^i) = W_g(U)$
and a sequence of hypersurfaces
$\{ \G_{t_i}^i \}$ with
$lim_{i \rightarrow \infty} \Hn(\G_{t_i}^i) \rightarrow W_g(U)$
will be called a min-max sequence.

Let $\mathcal{V}$ denote the space of varifolds
in $M$ with mass bounded by $2W_g(U)$.
$\mathcal{V}$ is endowed with weak* topology.
By the Riesz Representation Theorem and
the Banach-Alaoglu Theorem
this space is compact and metrizable. Let $\mathfrak{d}$ denote
a metric on $\mathcal{V}$ which induces this topology. 

Another important metric on the space of varifolds is 
given by (see \cite[2.1(19)]{Pi})

$${\bf F}(V_1,V_2) = \sup \{V_1(f)-V_2(f) | f \in \mathcal{K}(Gr_n(M)),
|f| \leq 1, Lip (f) \leq 1\}$$
where $\mathcal{K}(Gr_n(M))$ denotes the set of
Lipschitz functions compactly supported
in $Gr_n(M)$.

When manifold $M$ is compact 
the topology of the ${\bf F}$ metric and
the weak* topology on $\mathcal{V}$ coincide.
When $M$ is not compact these topologies
are different. 
Moreover, in this case $\mathcal{V}$
is not compact in the ${\bf F}$ metric.
The standard pull-tight argument 
(see \cite[Theorem 4.3]{Pi},  \cite[Proposition 4.1]{CD}
 and \cite[Proposition 8.5]{MN1}) uses compactness with 
the ${\bf F}$ metric in an important way, so in our case
the argument has to be modified. 
We apply the pull-tight iteratively 
on a sequence of nested open subsets $U_i$ exhausting $M$.
This is reminiscent of Schoen and Yau's 
arguments in the proof of Positive Mass Theorem \cite{SY}.
 
Let $\mathcal{V}_{st} \subset \mathcal{V}$ denote
the closed subset of stationary varifolds
in $\mathcal{V}$ (see \cite[8.2]{Si}).
If $\G$ is a hypersurface we 
will slightly abuse notation and
write $\G$ to denote the varifold induced by $\G$.

\begin{lemma} \label{pull-tight}
There exists a minimizing sequence
$\{\{ \G_{t}^i \}\}$
of good sweepouts of $U$,
such that for every min-max sequence
$\{ \G_{t_i}^i \}$ we have
$\lim_{i \rightarrow \infty} 
\mathfrak{d}(\G_{t_i}^i, \mathcal{V}_s) = 0$.
\end{lemma}

Let $\Om \subset M$ be an open subset.
Let $\mathcal{V}_{\Om}$ denote the space of varifolds
in $\Om$ with mass bounded by $2W_g(U)$.
For varifolds
in $\Om$ we can define metric

$${\bf F}_{\Om}(V_1, V_2)= \sup \{V_1(f)-V_2(f) | f \in \mathcal{K}(Gr_n(\Om)),
|f| \leq 1, Lip f \leq 1\}$$

It follows from the definition that
$${\bf F}_{\Om_1}(V_1 \llcorner Gr_n(\Om_1), 
V_2 \llcorner Gr_n(\Om_1)) \leq {\bf F}_{\Om_2}(V_1
\llcorner Gr_n(\Om_2), V_2 \llcorner Gr_n(\Om_2))$$
whenever $\Om_1 \subset \Om_2$.
When $\Om$ is a bounded subset of $M$
the weak topology on $\mathcal{V}_{\Om}$ and the topology induced
by the ${\bf F}_{\Om}$ metric coincide.

We will also need the following notation.
%
Let $\mathcal{V}_{\Om,st}$ denote the set of all stationary 
varifolds of $\Om$ of mass at most $2W_g$.

\begin{lemma} \label{local pull-tight}
Let $\Om_1 \subset \Om_2$ be two bounded open set.
There exists a map
$\Phi_{\Om_1, \Om_2}: \mathcal{V} \rightarrow \mathcal{V}$
and monotone sequences of positive numbers 
$\tau_1(\Om_1) \geq \tau_2(\Om_1) \geq ... \tau_k(\Om_1) \rightarrow 0$
and $\varepsilon_1(\Om_1) \geq \varepsilon_2(\Om_1) \geq ... \varepsilon_k(\Om_1) \rightarrow 0$
depending only on $\Om_1$ with the following properties.

\begin{enumerate}
    \item $||\Phi_{\Om_1, \Om_2}(V)||(M) \leq ||V||(M)$
    
    \item If $||V||(cl(\Om_2)) \leq \Cgs$
    then $\Phi_{\Om_1, \Om_2}(V)=V$
    
    \item If  $||V||(cl(\Om_2)) \geq \Clgs$
    and ${\bf F}_{\Om_1}(V \llcorner Gr_n(\Om_1), \mathcal{V}_{\Om_1,st}) \in [\frac{1}{2^{k+1}},\frac{1}{2^k}]$
    then the following holds:
    
    A. $||\Phi_{\Om_1, \Om_2}(V)||(M) \leq ||V||(M) - \varepsilon_k$
    
    B. ${\bf F} (V, \Phi_{\Om_1, \Om_2}(V)) \leq \tau_k$
\end{enumerate}

Moreover, if $\{ support(V_t) \}$ is a family
of hypersurfaces
in a sense of Definition \ref*{hypersurface_def}
then so is $\{ support(\Phi_{\Om_1, \Om_2}(V_t)) \}$.

\end{lemma}

\begin{proof}
Fix integer $k>0$.
Let $p(V) = V \llcorner Gr_n(\Om_1)$ denote the restriction function
and let $\mathcal{V}_{\Om_1,k}$ be the set of varifolds $V$
supported in $Gr_n(\Om_1)$ 
satisfying the following property:


(*) ${\bf F}_{\Om_1}(V, \mathcal{V}_{\Om_1,st}) \in [\frac{1}{2^{k+1}},\frac{1}{2^k}]$

It is straightforward to check that
$\mathcal{V}_{\Om_1,k}$ is compact in the topology induced by the ${\bf F}_{\Om_1}$ metric.

We will say that a smooth vector field $\chi$ is admissible if $\chi$
is compactly supported in $\Om_1$, 
$| \chi |_{C^1} \leq 1$ and $|\chi(x)| \leq dist(x, \partial \Om_1)$.
Let $X_{\Om_1}$ denote the set of all admissible vector fields.
We claim that there exists a $c_k >0$,
such that $\sup_{V \in \mathcal{V}_{\Om_1,k}} \inf_{\chi \in X_{\Om_1}}  \{ \delta V(\chi) \} < -c_k$
for otherwise there would exist a sequence of varifolds $V_i \in \mathcal{V}_{\Om_1,k}$ 
converging (in ${\bf F}_{\Om_1}$) to a stationary varifold supported in $Gr_n(\Om_1)$,
which contradicts condition (*) above.  Here, $\delta V(\chi)$ means the first variation
of $V$ with respect to the vector field $\chi$.

By compactness (cf. arguments in \cite[Theorem 4.3]{Pi},  \cite[Proposition 4.1]{CD}
 and \cite[Proposition 8.5]{MN1})
 we can find a locally finite open covering $\{ U^k_i \}$
of $\mathcal{V}_{\Om_1,k}$ and a collection of admissible vector fields
$\{ \chi^k_i \}$, such that $\delta V(\chi^k_i) < -\frac{c_k}{2}$
for all $V \in U^k_i$ and such that $ U^{k_1}_{i_1}$ is disjoint from  $U^{k_2}_{i_2}$
whenever $|k_1 - k_2| \geq 2$.

Choose a partition of unity $\{ \phi^k_i \}$ subordinate
to $\{ U^k_i \}$.
Let $\chi_V = \sum_{k,i} \phi^k_i(p(V)) \chi^k_i$.
We have that $\chi_{V}$ is admissible for all $V \in \mathcal{V}$
and $\delta V(\chi_{V}) < -\min \frac{1}{2}\{c_{k-1}, c_k, c_{k+1} \}$.

As in the proof of Proposition 4.1 in \cite{CD}
for each $V$ we can define a 1-parameter family of diffeomorphisms 
$\Psi_V: [0, \infty) \times M \rightarrow M$ with
$\frac{\partial \Psi_V (t,x)}{\partial t} = \chi_{V}(\Psi_V (t,x))$,
such that
$\Psi_V(0,x) = x$ for $t=0$ and $\Psi_V(t,x) = x$ for $x \in M \setminus \Om_1$.
It follows that we can make a continuous choice of $t = t(V)\leq 1/k$
and $\varepsilon_k>0$
so that $||\Psi_{V}(t_V, .)_{\#} V||(M) \leq ||V||(M) - \varepsilon_k$
for all $V \in p^{-1}(\mathcal{V}_{\Om_1,k})$.

Let $\eta: \mathcal{V} \rightarrow [0,1]$ be a continuous function
with $\eta(V) = 0$ if $||V ||(cl(\Om_2)) \leq \Cgs$
and $\eta(V) = 1$ if $||V ||(cl(\Om_2)) \geq \Clgs$.
We define $\Phi(V) = \Psi_{V}(\eta(V) t_V, .)_{ \#}(V)$.

Case (2) follows by definition of $\eta$
and Case (3) Properties A and B follow by construction.
\end{proof}

We use Lemma \ref*{local pull-tight}
to prove Lemma \ref*{pull-tight}.

Let $\{\{ \G_{t}^i \}\}$ be a minimizing
sequence of good sweepouts.
We will construct a minimizing sequence
of good sweepouts 
$\{\{ F_i(\G_{t}^i) \}\}$ satisfying the
conclusions of Lemma \ref*{pull-tight}.

Let $U_0 \subset U_1 \subset ...$ and $W_0 \subset W_1 \subset ...$
be bounded open sets with
$M = \bigcup U_i$, $U \subset U_i \subset W_i$ and
$\G_{t}^j \subset W_i$ for all $t$ and all $j \leq i$. 
Let $\tau^{U_i}_l$ and $\varepsilon^{U_i}_l$ be sequences
of numbers from Lemma \ref*{local pull-tight}
for $\Om_1 = U_i$ and $\Om_2 = W_i$.

Maps $F_i$'s are defined as follows.
We set $F_i(\G_{t}^i)$ to be the hypersurface
with $|F_i(\G_{t}^i)| = 
\Phi_{U_0,W_0} \circ ... \circ \Phi_{U_i,W_i}(|\G_{t}^i|)$,
where $\Phi_{U_i,W_i}$ is given by Lemma \ref*{local pull-tight}.
(Here we use the standard notation
that $|\Sigma|$ denotes the varifold induced
by hypersuface $\Sigma$).
Observe that by Lemma \ref*{local pull-tight}
$F_i(\G_{t}^i) = \G_{t}^i$ whenever $\Hn(\G_{t}^i) \leq \Cgs$.
Hence, $\{F_i(\G_{t}^i)\}$ is a good sweepout of $U$.

We claim 
if there exists a (not relabelled) subsequence
$\G_{t_j}^j$ with
$\lim_{j \rightarrow \infty} \Hn(F_j(\G_{t_j}^j)) 
= W_g$ then $\lim_{j \rightarrow \infty} 
{\bf F}_{U_i} (|F_j(\G_{t_j}^j)| \llcorner Gr_n(U_i),
\mathcal{V}_{U_i,st}) = 0$ for every $i$.
This implies Lemma \ref*{pull-tight}.

Fix $i$. For contradiction suppose there exists a 
(not relabelled)
subsequence $\{ F_j(\G_{t_j}^j) \}$ with 
$\lim_{j \rightarrow \infty} \Hn(F_j(\G_{t_j}^j)) 
= W_g$ and $\liminf_{j \rightarrow \infty} {\bf F}_{U_i}
(|F_j(\G_{t_j}^j)| \llcorner Gr_n(U_i),
\mathcal{V}_{U_i,st}) > \delta$.


Pick $k$ sufficiently large
so that $\frac{1}{2^k} + \sum_{l=0}^i \tau_k^{U_l}
< \delta/2$.
Let $\bar{\varepsilon} = \frac{1}{2} \min_{l=0,...,i} 
{\varepsilon^{l}_k}$.
Fix $j > i$ so that
$\Hn (\G_{t_j}^j) \in (W_g - \bar{\varepsilon}/10,
W_g + \bar{\varepsilon}/10)$. 

We have two possibilities. 
Suppose first that for some $l \in \{0,...,i\}$ the varifold
$V_l = \Phi_{U_{l+1}, W_{l+1}} \circ \Phi_{U_{l+2}, W_{l+2}} \circ... \circ \Phi_{U_j,W_j}(|\G_{t_j}^j|)$ satisfies
${\bf F}_{U_{l}}(V_l \llcorner Gr_n(U_{l}), \mathcal{V}_{U_{l},st}) > \frac{1}{2^k}$.
By Lemma \ref*{local pull-tight},
it follows that 
$$\Hn(F_j(\G_{t_j}^j))\leq ||\Phi_{U_{l}, W_l} (V_l)||(M)  \leq ||V_l||(M)  - \varepsilon_k^{U_l} 
\leq \Hn(\G_{t_j}^j)- \varepsilon_k^{U_l}<
W_g + \frac{\bar{\varepsilon}}{10}-\varepsilon_k^{U_l}<W_g - \bar{\varepsilon}$$
which contradicts our assumption on 
$\G_{t_j}^j$.

Suppose now that 
$V_l = \Phi_{U_{l+1}, W_{l+1}} \circ \Phi_{U_{l+2}, W_{l+2}} \circ... \circ \Phi_{U_j,W_j}(|\G_{t_j}^j|)$ satisfies
$${\bf F}_{U_{l}}(V_l \llcorner Gr_n(U_{l}), \mathcal{V}_{U_{l},st}) \leq \frac{1}{2^k}$$
for all $l \in \{0,...,i\}$.
We have that 
$$ {\bf F}_{U_{i}}(V_l \llcorner Gr_n(U_{i}),V_{l-1} \llcorner Gr_n(U_{i})) \leq {\bf F}(V_l,V_{l-1})
\leq \tau^l_k$$
By triangle inequality if follows that
${\bf F}_{U_{i}}(V_i \llcorner Gr_n(U_{i}),
F_j(\G_{t_j}^j) \llcorner Gr_n(U_{i}))
\leq \sum_{l=0} ^i \tau_k^l< \delta/2$.
From our choice of $k$ we obtain, as a result, that
${\bf F}_{U_{i}}(F_j(|\G_{t_j}^j|) \llcorner Gr_n(U_{i}), \mathcal{V}_{U_{i},st})< \delta$,
giving the desired contradiction.


%

\subsection{Almost minimizing hypersurfaces}

\begin{definition} (cf. \cite[3.2]{DT})
Given a pair of open sets $(U_1, U_2)$ we call a hypersurface 
$\Gamma$ $\varepsilon$-a.m. in $(U_1, U_2)$ if it is $\varepsilon$-a.m. in at least one of the two
open sets. Let $\mathcal{CO}(\mathcal{A})$ denote the set of pairs $(U_1, U_2)$ of open sets such that
$\inf_{x \in U_1, y \in U_2} d(x,y) \geq 4 \min \{diam(U_1), diam(U_2)\}$
and $U_i \in \mathcal{A}$ for $i=1,2$.
\end{definition}

Recall that $N_r(U)= \{x \in M: d(x,U) < r \}$ denotes the $r$-neighbourhood of $U$.
Let $\mathcal{A}(r,U)$ denote the set of all open subsets $V$ of $M$, such that
either $V \cap cl(U) = \emptyset$ or $V \subset N_r(U)$.

\begin{lemma} \label{combinatorial lemma}
Let $\{\{ \G_{t}^i \}\}$ be a minimizing sequence of good sweepouts as in Lemma \ref*{pull-tight}
and assume furthermore that $\Hn(\G_t^k)< W_g(U) + \frac{1}{8k}$.
For every $r>0$ and $N$ large enough, there exists $t_N \in [0,1]$ such that 
\begin{itemize}
\item $\G^N = \G^N_{t_N}$ 
is $\frac{1}{N}$-a.m. in all $(U_1, U_2) \in \mathcal{CO}(\mathcal{A}(r,U))$
\item $\Hn(\G^N) \geq W - \frac{1}{N}$
\item $\Hn(\G^N \cap cl(N_r(U))) \geq \varepsilon(U)/2$
\end{itemize}
\end{lemma}

\begin{proof}
The proof is by contradiction (cf. proofs of \cite[5.3]{CD} and \cite[3.4]{DT}).
Assume $N$ to be sufficiently large so that $\frac{1}{N}< \varepsilon/2$.
Let $A_N = \{t \in [0,1]: \Hn(\G_t ^N)\geq W_g(U) - \frac{1}{N} \}$ and 
$B_N(U,r) = \{t \in [0,1]: \Hn(\G_t ^N \cap cl(N_r(U)))\geq \varepsilon(U)/2 \}$
Define $K_N(U,r) = A_N \cap B_N(U,r)$.
 $K_N(U,r)$ is a compact set as $A_N$ and $B_N(U,r)$ are closed.
 By Proposition \ref*{prop:intersection_U} $K_N(U,r)$ is non-empty.

Assume the lemma to be false. Then there is a sequence $N_k$,
so that $\G_t^{N_k}$ is not $\frac{1}{N_k}$-a.m. in some pair $(U^1_t, U^2_t) \in \mathcal{CO}(\mathcal{A}(r,U))$
for every $t \in K_{N_k}(U,r)$. To simplify notation we will drop sub- and superscript $N_k$.
We will modify family $\G_t$ on some open set containing $K=K(U,r) \subset [0,1]$, so that the new family $\G'_t$ has
$\Hn(\G'_t) < W$ for all $\G'_t$ with $\Hn(\G'_t \cap U) > \varepsilon(U)$.

By Lemma 3.1 in \cite{DT} and refinement of the covering
argument on page 13 in \cite{DT} 
it is possible to choose a covering $J_i=(a_i,b_i)$ of $K$
and a collection of sets $U_i$ so that

\begin{itemize}
    \item each point of $K$ is contained in at most
    two intervals $J_i$
    \item  $U_i \in \mathcal{A}(r,U)$ for all $i$
    \item if $cl(J_i) \cap cl(J_j) \neq \emptyset$ then
    $\inf_{x \in U_i, y \in U_j} d(x,y) >0$
    \item there exists a $\delta>0$ such 
    that $\{(a_i+ \delta,b_i-\delta)\}$ still cover $K$ and
    a family $\{\Om_{i,t} \}$,
    such that 
    
    1) $\Om_{i,t} = \Om_t$ if $t \notin J_i$
    and $\Om_{i,t} \setminus U_i = \Om_{t} \setminus U_i$ 
    for all $t$; 
    
    2) $\Hn (\partial \Om_{i,t}) \leq 
    \Hn (\partial \Om_{t}) + \frac{1}{4N}$ for every $t$;
    
    3) $\Hn (\partial \Om_{i,t}) \leq 
    \Hn (\partial \Om_{t}) - \frac{1}{2N}$
    if $t \in (a_i+ \delta,b_i-\delta)$.
\end{itemize}

\begin{figure} 
	\centering
	\includegraphics[scale=0.7]{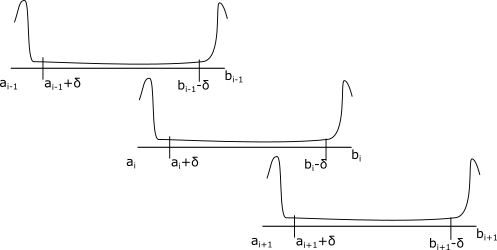}
	\caption{Graphs of areas of $\partial \Om_{i,t}$
	for overlapping intervals $[a_i,b_i]$}
	\label{fig:am_combinatorial}
\end{figure}

We define a new good sweepout $\{ \partial \Om'_t \}$ of $U$ given by

\begin{itemize}
\item $\Om'_t = \Om_t$ if $t \notin (a_i,b_i)$

\item $\Om'_t = \Om_{i,t}$ if $t$ is contained in a single $J_i$

\item $\Om'_t = [\Om_t \setminus (U_i \cup U_{i+1})] \cup [\Om_{i,t} \cap U_i]
\cup [\Om_{i+1,t} \cap U_{i+1}]$
if $t \in (a_i, b_i) \cap (a_{i+1}, b_{i+1})$
\end{itemize}

Observe that if $U_i \cap \partial U \neq \emptyset$ then
this modification of the family may lead to some 
transfer of area from $U$ to the complement of $U$.
However, in this case (since $U_i \subset N_r(U)$ 
by our definition of $\mathcal{A}(r,U)$)
the area of the surfaces inside $N_r(U)$ will not change.

\textbf{Claim:} 
If $ \Hn(\partial \Om'_t \cap U) \geq \varepsilon$ then 
$\Hn(\partial \Om'_t) < W_g(U)$. 

By Proposition \ref*{prop:intersection_U} the claim leads to the desired contradiction.

To prove the claim we verify several cases.

Case 1. Suppose $t \notin \bigcup J_i$ then, in particular, $t \notin K$ and
$\partial \Om'_t = \partial \Om_t$ satisfies $\Hn(\partial \Om'_t) < W - \frac{1}{N}$
or $\Hn(\partial \Om'_t \cap U) < \frac{\varepsilon}{2}$.

Case 2. Suppose $t \in \bigcup J_i$, but $t \notin K$.
We have two possibilities. Suppose first that 
$\partial \Om_t$ satisfies $\Hn(\partial \Om_t) < W - \frac{1}{N}$.
Since $t$ is contained in at most two distinct intervals $J_i$ we have that
$\Hn(\partial \Om'_t) \leq \Hn(\partial \Om_t) + 2 \frac{1}{4N}< W$.
So the claim holds.  

Suppose now that $\Hn(\partial \Om_t \cap N_r(U))< \varepsilon/2$.
We have that $t$ is contained in at most two intervals, say, $J_i$ and 
$J_{i+1}$. 
Family $\{\Om_t'\}$ only differs from $\{\Om_t\}$ inside $U_i \cup U_{i+1}$.
If the set $U_i \cup U_{i+1}$ is disjoint from 
$U$,
then $\Hn(\partial \Om_t' \cap U) = \Hn(\partial \Om_t \cap U) < \varepsilon/2$.

If $U_j$ (for $j=i$ or $i+1$) 
intersects $U$
then by definition of $\mathcal{A}(r,U)$ we must have 
$U_j \subset N_r(U)$ and so $\Hn(\partial \Om_t \cap U_j) < \varepsilon/2$.
It follows that 
$\Hn(\partial \Om'_t \cap U) \leq \Hn(\partial \Om_t \cap N_r(U)) + 2 \frac{1}{4N}
\leq  \varepsilon/2 + 2 \frac{1}{4N} < \varepsilon$.


Case 3. Suppose $t \in K$. Since the intervals $\{(a_i+ \delta,b_i-\delta)\}$
cover $K$ and each point of $K$ is contained in at most two intervals $J_i$ we have
that $\Hn(\partial \Om'_t) \leq \Hn(\partial \Om_t)+ \frac{1}{4N} - \frac{1}{2N} \leq W - \frac{1}{8N}$

\end{proof}

Now we can prove Proposition \ref*{sequence existence}. Fix $\delta>0$.
Let $\{ \G^N_{t_N} \}$ be the min-max sequence from Lemma \ref*{combinatorial lemma}. 
We will show that its subsequence satisfies the requirements of 
Proposition \ref*{sequence existence}.
Conditions (B) and (C) are satisfied by construction. 
We will choose a subsequence that also satisfies (A).

Observe that it follows from the definition
if $U \subset V$ and $\G$ is $\varepsilon$-a.m. in $V$
then $\G$ is $\varepsilon$-a.m. in $U$.

\textbf{Step 1. Almost minimizing annuli around points in $cl(U)$.}
We start by finding a subsequence of $\{ \G^N_{t_N} \}$ that is a.m. 
for annuli centered at $x \in cl(U)$.

By Lemma \ref*{combinatorial lemma} for each $0<r<\frac{\delta}{10}$ and each $x \in cl(U)$ we have that
$\G^k$ is $\frac{1}{k}$-a.m. either in $B_r(x)$ or $N_{1}(U) \setminus cl(B_{9r}(x))$.
For a fixed $r$ as above we have two possibilities.

(a) either $\{ \G^k \}$ is $1/k$-a.m. in $B_r(y)$ for $k>k(y)$ for all $y \in cl(U)$;

(b) or there is a (not relabeled) subsequence $\{ \G^k \}$ and a sequence
$\{ x^k_r \}$, $x^k_r \in cl(U)$, such that $ \G^k $ is 
$1/k$-a.m. in
$N_{1}(U) \setminus cl(B_{9r}(x^k_r)$.

Choose a sequence of radii $r_j \rightarrow 0$.  
If there exists $r_j>0$ such that (a) holds then 
condition (A) is satisfied for all $y \in cl(U)$ for
$r(y) = \min \{r_j,\delta \}$.
Suppose not. 
By compactness of $cl(U)$ we can select (not relabeled) subsequences
$x^k_{r_j} \rightarrow x^j \in cl(U)$
and $x^j \rightarrow x \in cl(U)$.
After choosing an appropriate diagonal subsequence we obtain
that $\G^k$ is $\frac{1}{k}$-a.m. in $N_{1}(U) \setminus cl(B_{\frac{1}{j}}(x))$
for all $k>j$. In particular, (A) of Proposition \ref*{regularity}
holds for all annuli centered at $x$ with $r(x)=\delta$.
For $y \in cl(U) \setminus x$ we obtain that $\{ \G^k \}$ is a.m. 
for annuli centered at $y$ with $r(y)= \min \{\delta, d(y,x) \}$.

\vspace{3mm}

\textbf{Step 2. Almost minimizing annuli around points in 
$M \setminus cl(U)$.}
Let $\{ \G^n \}$ denote the min-max sequence from Step 1.
By Lemma \ref*{combinatorial lemma} for each $y \in M \setminus cl(U)$
we have that 

(a) either $\{ \G^k \}$ is $1/k$-a.m. in $B_r(y) \setminus cl(U)$ for $k>k(y)$ for all $y \in M \setminus cl(U)$;

(b) or there is a (not relabeled) subsequence $\{ \G^k \}$ and a sequence
$\{ x^k_r \}$, $x^k_r \in M \setminus cl(U)$, such that $ \G^k $ is 
$1/k$-a.m. in
$M \setminus cl(U \cup B_{9r}(x^k_r))$.

If (a) holds for some positive radius $r_0$ then condition (A) is satisfied
for all $y \in M \setminus cl(U)$ for $r(y) = \min \{r_0, d(y, cl(U) \}$.
Otherwise, we obtain a sequence $\{ x^j \}$ and a (not relabeled) subsequence
$\{ \G^k \}$, such that
that $\G^j$ is $\frac{1}{j}$-a.m. in $M \setminus cl(U \cup B_{1/j}(x^j))$
for all large $j$. If sequence $\{ x_j \}$ contains a subsequence that converges to a point
$x \in M \setminus cl(U)$ then we verify that for a subsequence of $\{ \G^k \}$ condition (A) is
satisfied for $x$ with $r(x) = d(x,cl(U)$ and for all $y \in M \setminus cl(U)$ with $r(y) = \min \{ d(y,cl(U)), d(y,x) \}$.
Otherwise there is a subsequence of $\{ x_j \}$, such that either $d(x_j, cl(U)) \rightarrow \infty$
or $d(x_j, cl(U)) \rightarrow 0$. In both cases we have that 
condition (A) is
satisfied for all $y \in M \setminus cl(U)$ with $r(y) = d(y,cl(U))$.

%
%
%
%

%


\begin{tabbing}
\hspace*{7.5cm}\=\kill
Gregory R. Chambers                 \> Yevgeny Liokumovich\\
Department of Mathematics           \> Department of Mathematics\\
Rice University                     \> Massachusetts Institute of Technology\\
Houston, Texas 77005             \> Cambridge, MA 02142\\
USA                                 \> USA\\
e-mail: gchambers@rice.edu      \> e-mail: ylio@mit.edu\\
\end{tabbing}


\begin{thebibliography}{99}
\bibliographystyle{alpha}
\small

\bibitem[Al]{Al} F. Almgren, The theory of varifolds, Mimeographed notes, Princeton (1965)
\bibitem[Ba]{Ba} V. Bangert, Closed geodesics on complete surfaces, Math. Ann. 251 (1980) 83-96.
\bibitem[DT]{DT} C. De Lellis and D. Tasnady. The existence of embedded minimal hypersurfaces,
J. Differential Geom. 95 (2013), no. 3, 355-388.
\bibitem[CHMR]{CHMR} P. Collin, L. Hauswirth,
L. Mazet, H. Rosenberg, Minimal surfaces in finite volume non compact hyperbolic 3-manifolds, arXiv:1405.1324
\bibitem[Gr]{Gr}  M. Gromov, Plateau-Stein manifolds, Cent. Eur. J. Math., 12(7):923-951, 2014.
\bibitem[Fa]{Fa}  Falconer, K. J., Continuity properties of k-plane integrals and Besicovitch sets, Math. Proc.
Cambridge Phil. Soc. 87 (1980) no. 2, 221-226.
\bibitem[CR]{CR} G.R. Chambers, R. Rotman, Contracting loops on a Riemannian 2-surface, preprint.
\bibitem[CD]{CD}  T. Colding and C. De Lellis, The min-max construction of minimal surfaces, Surveys in differential
geometry, Vol. VIII (Boston, MA, 2002), 75-107, Int. Press, Somerville, MA, 2003.
\bibitem[GL]{GL}  P. Glynn-Adey and Y. Liokumovich. Width, Ricci curvature and minimal hypersurfaces, 
to appear in  J. Differential Geom.
\bibitem[GW]{GW} R. E. Greene, H. Wu, $C^\infty$ approximation of convex, subharmonic and plurisubharmonic functions, Annales scientifiques de l'ENS, 1979.
\bibitem[Gu1]{Gu1} L. Guth, The width-volume inequality, Geom. Funct. Anal. 17 (2007), 1139-1179.
\bibitem[KZ]{KZ} D. Ketover, X. Zhou, Entropy of closed surfaces and min-max theory, preprint.
\bibitem[MR]{MR} L. Mazet and H. Rosenberg, Minimal hypersurfaces of least area, arXiv:1503.02938v2
\bibitem[MN1]{MN1} F.C. Marques, A. Neves, Min-max theory and the Willmore conjecture, Ann. of Math. 179 (2014) 683-782.
\bibitem[MN2]{MN2} F.C. Marques, A. Neves, Morse index and multiplicity of min-max minimal hypersurfaces,
arxiv:1512.06460.
\bibitem[Mi1]{Mi1} J. Milnor, Morse theory, Princeton, 1963.
\bibitem[Mi2]{Mi2} J. Milnor, Lectures on the h-cobordism Theorem, Princeton, 1965.
\bibitem[Mo]{Mo} R. Montezuma, Min-max minimal hypersurfaces in non-compact manifolds, J. Differential Geom.
Volume 103, Number 3 (2016), 475-519.
\bibitem[Mu]{Mu} A. Mukherjee, Differential Topology, Springer, 2015.
\bibitem[Pi]{Pi} J. Pitts, Existence and regularity of minimal surfaces on Riemannian manifold, Mathematical
Notes 27, Princeton University Press, Princeton 1981.
\bibitem[Sa]{Sa} S. Sabourau, Volume of minimal hypersurfaces in manifolds with nonnegative Ricci curvature, J. reine angew. Math., to appear
\bibitem[SY]{SY} R. Schoen and S.-T. Yau, On the proof of the positive mass conjecture in general relativity, Comm. Math. Phys. 65 (1979), no. 1, 45-76.
\bibitem[So1]{So1}  A. Song, Embeddedness of least area minimal hypersurfaces, arXiv:1511.02844
\bibitem[So2]{So2}  A. Song, Existence of infinitely many minimal hypersurfaces in closed manifold,. arXiv:1806.08816.
\bibitem[Si]{Si} L. Simon, Lectures on geometric measure theory, Australian National University Centre for
Mathematical Analysis, Canberra, 1983.
\bibitem[Th]{Th}  G. Thorbergsson, Closed geodesics on non-compact Riemannian manifolds, Math.Z. 159 (1978), 249-258.
\bibitem[Zh1]{Zh1} X. Zhou, Min-max minimal hypersurface in $(M^{n+1}, g)$ with $Ric>0$ and $2 \leq n \leq 6$,
J. Differential Geom. 100 (2015), no. 1, 129-160.
\bibitem[Zh2]{Zh2} X. Zhou, Min-max hypersurface in manifold of positive Ricci curvature, arXiv:1504.00966.
\end{thebibliography}
\end{document}